\documentclass[12pt,reqno]{amsart}
\usepackage[centering,margin=1in]{geometry}
\usepackage{amssymb,amsfonts,amsmath,mathrsfs,etoolbox,mathtools,accents,array,bm,bbm,booktabs,textcomp,xspace}
\usepackage[utf8]{inputenc}
\usepackage[UKenglish,iso]{isodate}
\usepackage{enumerate}
\usepackage[dvipsnames]{xcolor}
\definecolor{dark-blue}{rgb}{0.15,0.15,0.4}
\definecolor{dark-red}{rgb}{0.4,0.15,0.15}
\definecolor{medium-blue}{rgb}{0,0,0.5}
\usepackage[breaklinks=true,linktocpage=true,colorlinks=true,linktoc=all,linkcolor=dark-blue,citecolor=dark-blue,filecolor=dark-blue,urlcolor=medium-blue,backref=page,hypertexnames=false]{hyperref}
\usepackage[nameinlink,capitalize]{cleveref}
\usepackage{autonum}
\usepackage{constants}
\newconstantfamily{constant}{symbol=c}
\usepackage{orcidlink}
\allowdisplaybreaks

\crefname{lemma}{Lemma}{Lemmata}
\Crefname{lemma}{Lemma}{Lemmata}

\crefmultiformat{theorem}{Theorems~#2#1#3}{ and~#2#1#3}{,~#2#1#3}{,~and~#2#1#3}
\crefmultiformat{lemma}{Lemmata~#2#1#3}{ and~#2#1#3}{,~#2#1#3}{,~and~#2#1#3}
\crefmultiformat{proposition}{Propositions~#2#1#3}{ and~#2#1#3}{,~#2#1#3}{,~and~#2#1#3}

\usepackage{cite}
\makeatletter
\newcommand{\citecomment}[2][]{\citen{#2}#1\citevar}
\newcommand{\citeone}[1]{\citecomment{#1}}
\newcommand{\citetwo}[2][]{\citecomment[, #1]{#2}}
\newcommand{\citevar}{\@ifnextchar\bgroup{; \citeone}{\@ifnextchar[{; \citetwo}{]\xspace}}}
\newcommand{\citefirst}{\@ifnextchar\bgroup{\citeone}{\@ifnextchar[{\citetwo}{]}}}
\newcommand{\cites}{[\citefirst}
\makeatother

\newtheorem{theorem}{Theorem}[section]

\newtheorem{definition}[theorem]{Definition}
\newtheorem{lemma}[theorem]{Lemma}
\newtheorem{proposition}[theorem]{Proposition}
\theoremstyle{definition}

\numberwithin{equation}{section}

\renewcommand{\Re}{\mathrm{Re}}
\renewcommand{\Im}{\mathrm{Im}}
\renewcommand{\epsilon}{\varepsilon}
\patchcmd{\section}{\scshape}{\bfseries}{}{}
\makeatletter
\renewcommand{\@secnumfont}{\bfseries}
\makeatother
\newcommand{\norm}[1]{\left\lVert#1\right\rVert}
\makeatletter\newcommand{\tpmod}[1]{{\@displayfalse \pmod{#1}}}
\renewcommand*\backref[1]{$\uparrow$\thinspace\ifx#1\relax \else #1 \fi}
\newcommand{\MYhref}[3][dark-red]{\href{#2}{\color{#1}{#3}}}

\begin{document}

\title[A Subconvex Metaplectic PGT and the Shimura Correspondence]{A Subconvex Metaplectic Prime Geodesic Theorem and the Shimura Correspondence}

\author[Ikuya Kaneko]{Ikuya Kaneko\,\orcidlink{0000-0003-4518-1805}}

\subjclass[2020]{Primary 11F27, 11F72; Secondary 11F30, 11M36}

\keywords{Prime geodesic theorem, varying weight, metaplectic covering, Kubota character, Shimura correspondence, explicit formula}

\cleanlookdateon

\date{\today~(ISO 8601)}

\begin{abstract}
We investigate the prime geodesic theorem with an error term dependent on the varying weight and its higher metaplectic coverings in the arithmetic setting, each admitting subconvex refinements despite the softness of our input. The former breaks the $\frac{3}{4}$-barrier due to Hejhal (1983) when the multiplier system is nontrivial, while the latter represents the~first theoretical evidence supporting the prevailing consensus on the optimal exponent $1+\varepsilon$ when the multiplier system specialises to the Kubota character. Our argument relies on the elegant phenomenon that the main term in the prime geodesic theorem is governed by the size of the largest residual Laplace eigenvalue, thereby yielding a simultaneous polynomial~power-saving in the error term relative to its Shimura correspondent where the multiplier system~is~trivial.
\end{abstract}

\maketitle

\section{Introduction and statement of results}
Let $\mathcal{M}$ be a smooth, connected Riemannian orbifold. A~\textit{closed geodesic} on $\mathcal{M}$ is a geodesic that returns to its starting point with the same tangent direction, constituting a fundamental geometric and dynamical invariant. Closed geodesics underpin periodic orbits of the geodesic flow, while the systole acts as a major threshold in global geometry and dynamics. Notably, it is well known that the number of simple closed geodesics of length at most a prescribed value grows polynomially~\cites{Mirzakhani2004}{Mirzakhani2008}{Rivin2001}, while the total number of closed geodesics~grows exponentially~\cites{Margulis1969}{PollicottSharp1998}{Sinai1966}. On locally symmetric spaces, the trace formula provides a bridge between closed geodesics and the spectrum of the Casimir operator, which quantises the geodesic flow. If $\mathcal{M}$ is arithmetic in nature, then the lengths and multiplicities of closed geodesics encode intrinsic relationships with the asymptotic properties of class numbers and regulators of binary quadratic forms~\cites{Sarnak1982}{Sarnak1985}, thereby fostering a more comprehensive and deeper understanding of closed geodesics on such $\mathcal{M}$ at the confluence of various branches of mathematics, including dynamical systems, ergodic theory, geometry, and number theory.

Let $\mathrm{G} \coloneqq \mathrm{SO}_{0}(d, 1)$ be a connected, noncompact, semisimple Lie group of real rank $1$, let $\mathrm{K}$ be the maximal compact subgroup of $\mathrm{G}$, and let $\Gamma$ be a discrete, torsion-free subgroup of $\mathrm{G}$. For ease of exposition, we consider in greater generality the $d$-dimensional locally symmetric space
\begin{equation}\label{eq:locally-symmetric-space}
\mathbb{X}_{d} \coloneqq \Gamma \backslash \mathbb{H}_{d} \cong \Gamma \backslash \mathrm{G}/\mathrm{K}.
\end{equation}
There exists a bijective correspondence between a closed geodesic $C_{\gamma}$ and a conjugacy class~of $\gamma \in \Gamma$. It is convenient to denote its length by $l(\gamma)$ and its norm by $\mathrm{N}(\gamma) \coloneqq e^{l(\gamma)}$, to which one may attach the hyperbolic analogue $\Lambda_{\Gamma}(\gamma)$ of the von Mangoldt function. If $\nu: \Gamma \to \mathrm{GL}_{D}(\mathbb{C})$ is a unitary representation, then the \textit{twisted Chebyshev-like counting function} is defined by
\begin{equation}\label{eq:Psi}
\Psi_{\Gamma}^{(d)}(x, \nu) \coloneqq \sum_{\mathrm{N}(\gamma) \leq x} \mathrm{tr}(\nu(\gamma)) \Lambda_{\Gamma}(\gamma), \qquad \Psi_{\Gamma}^{(d)}(x) \coloneqq \Psi_{\Gamma}^{(d)}(x, \mathbbm{1}).
\end{equation}
Understanding the asymptotic behaviour of $\Psi_{\Gamma}^{(d)}(x, \nu)$ falls under the umbrella of vast swaths of investigations into the \textit{prime geodesic theorem}, which has garnered distinguished attention among number theorists over the past several decades. The structure of the present paper is tripartite, with each part devoted to a distinct yet theoretically interrelated topic as follows:
{\begin{enumerate}[(i)]
\item A vector-valued twisted prime geodesic theorem valid for any cofinite Fuchsian group;
\item An arithmetic refinement in weight $\frac{1}{2}$, in conjunction with the Shimura correspondence;
\item A subconvex $n$-fold metaplectic prime geodesic theorem extending to fractional weights.
\end{enumerate}

\subsection{Varying weight}
As a precursor, for $d = 2$ and $\nu = \mathbbm{1}$, Selberg laid the cornerstone for the prime geodesic theorem, with subsequent extensions to the case of $\nu \ne \mathbbm{1}$ by Hejhal~\cites[Theorem~3.5]{Hejhal1983}. Let $\Gamma \subset \mathrm{PSL}_{2}(\mathbb{R})$, and let $\nu: \Gamma \to \mathrm{GL}_{D}(\mathbb{C})$ be a unitary character of weight $k \in (-1, 1]$ and dimension $D \geq 1$, for which the existence of a continuous family of multiplier systems $\nu$ and automorphic forms that transform with respect to $\nu$ is guaranteed by~\cites[Page~534]{Petersson1938}; cf.~\cites[Theorem~5.1]{BurrinvonEssen2024}[Proposition~2.2]{Hejhal1983}. Then the pseudoprime counting function $\Psi_{\Gamma}^{(2)}(x, \nu)$ obeys an asymptotic formula exhibiting a structural resemblance to that of the standard prime counting function, namely if $\mathcal{L}_{k}(\Gamma, \nu)$ stands for finite-dimensional~Hilbert space of square-integrable automorphic forms of weight $k$ and multiplier system $\nu$ on $\Gamma$, then
\begin{equation}\label{eq:Hejhal}
\Psi_{\Gamma}^{(2)}(x, \nu) = \mathop{\sum \sum}_{\substack{f \in \mathcal{L}_{k}(\Gamma, \nu) \\ \frac{1}{2} < s_{f} \leq 1-\frac{|k|}{2}}} \frac{x^{s_{f}}}{s_{f}}+\mathcal{E}_{\Gamma}^{(2)}(x, \nu),
\end{equation}
where the inner summation is taken over the small eigenvalues $\lambda_{f} = s_{f}(1-s_{f}) \in [\frac{|k|}{2}(1-\frac{|k|}{2}), \frac{1}{4})$ of the Laplacian on $\Gamma \backslash \mathbb{H}_{2}$, with the convention that the sum is considered empty for $k = -1$, whereas $\mathcal{E}_{\Gamma}^{(2)}(x, \nu)$ serves as an error term satisfying~\cites[Theorem~7]{BurrinvonEssen2024}[Theorem~3.4]{Hejhal1983}
\begin{equation}\label{eq:3/4}
\mathcal{E}_{\Gamma}^{(2)}(x, \nu) \ll_{\Gamma, \varepsilon} x^{\frac{3}{4}+\varepsilon}.
\end{equation}
Any progress beyond the barrier~\eqref{eq:3/4} has proven elusive in full generality. Furthermore, the optimal exponent is conjectured to be $\frac{1}{2}+\varepsilon$ due to the validity of an analogue of the~\textit{Riemann hypothesis} for the Selberg zeta function apart from at most a finite number of the exceptional zeros. It follows from the $\Omega$-result of Hejhal~\cites[Theorem~3.8]{Hejhal1983} that the conjecture is sharp up to an arbitrarily small quantity $\varepsilon > 0$, although it remains well beyond the reach of~current technology; see~\cites[Page~40]{Sarnak1980} for what appears to be its earliest, albeit purely speculative, consideration of this kind, consistent with the folklore philosophy of prime number theory.
\begin{definition}\label{def:delta-k-nu}
Let $\Gamma \subset \mathrm{PSL}_{2}(\mathbb{R})$, and let $\nu: \Gamma \to \mathrm{GL}_{D}(\mathbb{C})$ be a unitary multiplier system of weight $k \in (-1, 1]$ and dimension $D$. We define $\delta_{1}^{(2)}(\nu) \in [\frac{1}{2}, \frac{3}{4}]$ so that
\begin{equation}
\mathcal{E}_{\Gamma}^{(2)}(x, \nu) \ll_{\Gamma, \nu, \varepsilon} x^{\delta_{1}^{(2)}(\nu)+\varepsilon},
\end{equation}
while we define $\delta_{2}^{(2)}(\nu) \in [\frac{1}{2}, \frac{2}{3}]$ so that there exists a constant $\eta_{2}^{(2)}(\nu) \geq 0$ such that
\begin{equation}\label{eq:def-2}
\Big(\frac{1}{Y} \int_{X}^{X+Y} |\mathcal{E}_{\Gamma}^{(2)}(x, \nu)|^{2} \, dx \Big)^{\frac{1}{2}} \ll_{\Gamma, \nu, \varepsilon} X^{\delta_{2}^{(2)}(\nu)+\varepsilon} \Big(\frac{X}{Y} \Big)^{\eta_{2}^{(2)}(\nu)}.
\end{equation}
\end{definition}

Our first main result marks the first polynomial power-saving refinement over~\eqref{eq:3/4}, which had persisted as the best known unconditional result in such generality for over four decades.
\begin{theorem}\label{thm:main}
Keep the notation as in \cref{def:delta-k-nu}. Then unconditionally
\begin{equation}\label{eq:beyond-3/4}
\delta_{1}^{(2)}(\nu) \leq \frac{3}{4}-\frac{|k|}{4}.
\end{equation}
\end{theorem}

\cref{thm:main} recovers~\eqref{eq:3/4} for $k = 0$. The argument relies fundamentally on induction with respect to $\delta_{1}^{(2)}(\nu)$ and a crude form of the \textit{Brun--Titchmarsh-type theorem} over short~intervals (\cref{prop:Brun-Titchmarsh}), which in turn facilitates a straightforward derivation of the \textit{explicit formula} (\cref{prop:spectral-explicit-formula}) valid for any cofinite Fuchsian group $\Gamma \subset \mathrm{PSL}_{2}(\mathbb{R})$. As shall be discussed in due course, substituting $k = \frac{1}{2}$ into~\eqref{eq:beyond-3/4} leads to a one-line proof of~\cites[Theorem~3.1]{Matthes1994}.

Furthermore, the second moment theory of the prime geodesic theorem in the $2$-dimensional setting was undertaken by Cherubini and Guerreiro~\cites[Theorems~1.1]{CherubiniGuerreiro2018}, with subsequent unconditional improvements in the arithmetic setting by Balog et al.~\cites[Theorem~1.1]{BalogBiroHarcosMaga2019} and by the author~\cites[Theorem~1.1]{Kaneko2020}. An adaptation of the machinery of Cherubini and Guerreiro via the vector-valued Selberg trace formula (\cref{thm:Selberg-trace-formula}) leads to a power-saving refinement of~\cites[Theorems~1.1]{CherubiniGuerreiro2018} when the underlying multiplier system is nontrivial.
\begin{theorem}\label{thm:CG}
Keep the notation as in \cref{def:delta-k-nu}. Then unconditionally
\begin{equation}\label{eq:CG}
\delta_{2}^{(2)}(\nu) \leq \frac{2}{3}-\frac{|k|}{6}, \qquad \eta_{2}^{(2)}(\nu) \leq \frac{1}{3}.
\end{equation}
\end{theorem}

\cref{thm:CG} aligns with \cref{thm:main} when substituted into the mean-to-max result.
\begin{theorem}\label{thm:mean-to-max}
Keep the notation as in \cref{def:delta-k-nu}. Then unconditionally
\begin{equation}\label{eq:mean-to-max}
\delta_{1}^{(2)}(\nu) \leq \frac{\delta_{2}^{(2)}(\nu)+(1-\frac{|k|}{2}) \eta_{2}^{(2)}(\nu)}{1+\eta_{2}^{(2)}(\nu)}.
\end{equation}
\end{theorem}

\subsection{Shimura correspondence}\label{subsect:Shimura}
If $\Gamma$ is arithmetic and $\nu$ is trivial, then the $\frac{3}{4}$-barrier~\eqref{eq:3/4} is surpassed in the prior breakthroughs~\cites{Cai2002}{Iwaniec1984}{Kaneko2025-4}{LuoSarnak1995}{LuoRudnickSarnak1995}{Koyama1998}{SoundararajanYoung2013}. Why do arithmetic orbifolds admit a refinement of the error term in~\eqref{eq:3/4}? In alignment~with the spectral deformation theory due to Phillips and Sarnak~\cites{PhillipsSarnak1985-2}{PhillipsSarnak1985}, a key dichotomy emerges between the arithmetic and non-arithmetic settings, predicated upon the observation that, for a generic cofinite $\Gamma \subset \mathrm{PSL}_{2}(\mathbb{R})$, the existence of Maa{\ss} forms is severely constrained unless the underlying structure enjoys certain arithmetic or geometric symmetries.

We delve into an analogue of the prime geodesic theorem where $\nu$ specialises to the $3$-fold theta multiplier system, thereby causing the asymptotic~\eqref{eq:Hejhal} to possess a main term~smaller than $x$. In accordance with the reduction of the main term, the prevailing consensus suggests that the error term could likewise be rendered smaller than in~\eqref{eq:beyond-3/4} with the aid of additional oscillations in sums of half-integral weight Kloosterman sums. To begin with, we define
\begin{equation}
\vartheta_{2}(z) \coloneqq \sum_{n \in \mathbb{Z}} e^{\pi i(n+\frac{1}{2})^{2} z}, \qquad \vartheta_{3}(z) \coloneqq \sum_{n \in \mathbb{Z}} e^{\pi in^{2} z}, \qquad \vartheta_{4}(z) \coloneqq \vartheta_{3}(z+1).
\end{equation}
Then the vector $\Theta(z) \coloneqq (\vartheta_{2}(z), \vartheta_{3}(z), \vartheta_{4}(z))^{\mathrm{T}}$ is a holomorphic modular form of weight $\frac{1}{2}$ on $\Gamma = \mathrm{SL}_{2}(\mathbb{Z})$, and the multiplier system $\nu_{\Theta}: \Gamma \mapsto \mathrm{Aut}(\mathbb{C}^{3})$ is called the \textit{$3$-fold theta~multiplier}. This introduces a difference from the classical theta multiplier on $\Gamma_{0}(4)$. Notably, under the Shimura correspondence~\cites{Shimura1973}, the Kohnen plus spaces of scalar-valued half-integral~weight forms on $\Gamma_{0}(4)$ are isomorphic as Hecke modules to spaces of weight $0$ cusp forms on $\Gamma$,~while the full spaces on $\Gamma_{0}(4)$ lift to $\Gamma_{0}(2)$. Nonetheless, the vector-valued nature of the multiplier system $\nu_{\Theta}$ obviates the need for such a lift, allowing the correspondence to be realised directly on $\Gamma$. This provides a more intrinsic perspective without restricting the underlying space;~see \cref{thm:Shimura-vector-valued} for the transition between the spectral parameters.

The following result represents the half-integral weight counterpart of~\cites[Theorem~1.4]{LuoSarnak1995}. The optimisation of the resulting exponent in the spirit of~\cites[Theorem~1.1]{SoundararajanYoung2013} is deferred to future pursuits, as it requires strong arithmetic and geometric input in a brute force~manner.
\begin{theorem}\label{thm:pointwise}
Let $\nu_{\Theta}$ denote the $3$-fold theta multiplier. Then unconditionally
\begin{equation}\label{eq:nu-Theta}
\delta_{1}^{(2)}(\nu_{\Theta}) \leq \frac{3}{5} = 0.6.
\end{equation}
\end{theorem}
 
Although endeavours towards the twisted prime geodesic theorem from a number-theoretic perspective are rare and sporadic, the prime geodesic theorem associated to the $3$-fold theta multiplier system holds particular significance, both because of its intrinsic attributes and as a signpost towards a general framework. Investigations in this direction can be traced~back~to Barner~\cites[Equation~(6.6)]{Barner1992}, who proposes a crude asymptotic formula
\begin{equation}\label{eq:Barner}
\Psi_{\Gamma}^{(2)}(x, \nu_{\Theta}) \sim \frac{4}{3} x^{\frac{3}{4}}.
\end{equation}
The bottom Laplace eigenvalue is $\frac{3}{16}$ instead of $0$, and \cref{thm:main} precludes the possibility of~\eqref{eq:Barner} degenerating into a mere upper bound. Numerical evidence~for~\eqref{eq:Barner} is provided by Barner~\cites[Section~6]{Barner1992} via an efficient algorithm for computing the pseudoprime counting function up to $x \in [2, 2250000]$. Shortly thereafter, Matthes~\cites[Main Theorem]{Matthes1994}[Theorem~1.1]{Matthes1996}\footnote{The result was initially presented in his Habilitationsschrift defended in 1992 at Gesamthochschule~Kassel, a condensed version of which is published as~\cites{Matthes1994}.} adapted the toolbox of Iwaniec~\cites{Iwaniec1984} to strengthen~\eqref{eq:Barner} with an error term. Such a reduction of the error term is heuristically foreseeable for a certain family of multiplier systems wherein the \textit{Shimura correspondence} is available in a highly explicit form.

\begingroup
\renewcommand{\arraystretch}{1.5}
\begin{table}[ht!]
\centering
\setlength\tabcolsep{10pt}
\caption{A Shimura correspondence between $\delta_{1}^{(2)}(\mathbbm{1})$ and $\delta_{1}^{(2)}(\nu_{\Theta})$}
\label{table:1}
\begin{tabular}{|c|c|c|c|}
	\hline
	Source & $\delta_{1}^{(2)}(\mathbbm{1})$ & $\delta_{1}^{(2)}(\nu_{\Theta})$ & $\delta_{1}^{(2)}(\mathbbm{1})-2(\delta_{1}^{(2)}(\nu_{\Theta})-\frac{1}{2})$ \\ \hline
	Trivial & $1$ & $\frac{3}{4}$ & $\frac{1}{2}$ \\ \hline
	\cites{Huber1961}{Matthes1994} & $\frac{3}{4}$ & $\frac{5}{8}$ & $\frac{1}{2}$ \\ \hline
	\cites{Iwaniec1984}{Matthes1996} & $\frac{35}{48}$ & $\frac{59}{96}$ & $\frac{1}{2}$\\ \hline
	\cites{LuoSarnak1995} \& \eqref{eq:nu-Theta} & $\frac{7}{10}$ & $\frac{3}{5}$ & $\frac{1}{2}$\\ \hline
\end{tabular}
\end{table}
\endgroup

To reveal the hidden structure, we assemble in \cref{table:1} several previous pointwise exponents for the prime geodesic theorems in weight $0$ and $\frac{1}{2}$. The correspondence is expected to persist for the second moment exponents. As straightforward evidence, it follows -- albeit with some additional technicalities -- from the treatment in~\cites[Section~4]{CherubiniGuerreiro2018} that
\begin{equation}
\delta_{2}^{(2)}(\nu_{\Theta}) \leq \frac{9}{16}, \qquad \eta_{2}^{(2)}(\nu_{\Theta}) \leq \frac{1}{4},
\end{equation}
thereby reproducing \cref{thm:pointwise} when substituted into \cref{thm:mean-to-max}.

Note that the specialisation to $\nu_{\Theta}$ serves to simplify our proofs; however, the methodologies used herein are applicable to a broader family of sufficiently well-behaved multiplier systems. For example, if one classifies the admissible multiplier systems in~\cites[Definition~1.1]{Sun2025},~then any quadratic twist of $\nu_{\Theta}$ or its conjugate becomes admissible. Hence, we are led to the prime geodesic theorem wherein the underlying multiplier system is induced by the theta multiplier on $\Gamma_{0}(4q)$ with $q \in \mathbb{N}$. Since neither the number of inequivalent cusps nor the lack of specific knowledge regarding admissible multiplier systems exerts any theoretical bearing on the core of our argument, one may derive the same result~\eqref{eq:nu-Theta} for such a family of multiplier systems.

\subsection{Metaplectic coverings}
For $\Gamma = \mathrm{PSL}_{2}(\mathbb{Z})$, the permissible values of $k$ in \cref{thm:main} are $0$ and $\pm \frac{1}{2}$. To accommodate the fractional weight $k = \frac{\cdot}{n}$, it is necessary to investigate the $n$-fold metaplectic covering. The full modular group itself is insufficient to define automorphic forms of fractional~weight since it introduces multi-valued automorphy that is not well-defined without a covering group. Consequently, it is natural to~extend our consideration to the case of $d = 3$ in~\eqref{eq:locally-symmetric-space} and to examine the corresponding prime geodesic theorem, with particular emphasis on the underlying arithmetic structure and its relationship with the trivial covering.

For $\nu = \mathbbm{1}$, the first effective asymptotic formula for $\Psi_{\Gamma}^{(3)}(x)$ is attributed to Sarnak~\cites[Theorem~5.1]{Sarnak1983} and Nakasuji~\cites[Theorem~1.3]{Nakasuji2000}[Theorem~4.2]{Nakasuji2001}, which demonstrates in particular that for any cofinite Kleinian group $\Gamma \subset \mathrm{PSL}_{2}(\mathbb{C})$,
\begin{equation}\label{eq:Sarnak}
\Psi_{\Gamma}^{(3)}(x) = \mathop{\sum \sum}_{\substack{f \in \mathcal{L}(\Gamma, \mathbbm{1}) \\ 1 < s_{f} \leq 2}} \frac{x^{s_{f}}}{s_{f}}+\mathcal{E}_{\Gamma}^{(3)}(x),
\end{equation}
where the inner summation is taken over the small eigenvalues $\lambda_{f} = s_{f}(2-s_{f}) \in [0, 1)$ of~the Laplacian on $\Gamma \backslash \mathbb{H}_{3}$, whereas $\mathcal{E}_{\Gamma}^{(3)}(x)$ serves as an error term satisfying
\begin{equation}\label{eq:5/3}
\mathcal{E}_{\Gamma}^{(3)}(x) \ll_{\Gamma, \varepsilon} x^{\frac{5}{3}+\varepsilon}.
\end{equation}
Any progress beyond the barrier~\eqref{eq:5/3} has proven elusive in full generality, while polynomial power-saving refinements exist in some restrictive arithmetic frameworks~\cites{BalogBiroCherubiniLaaksonen2022}{BalkanovaChatzakosCherubiniFrolenkovLaaksonen2019}{BalkanovaFrolenkov2020}{BalkanovaFrolenkov2022}{Kaneko2022-2}{Kaneko2025-2}{Koyama2001}{Qi2024}. The optimal exponent is conjectured to be~$1+\varepsilon$ as indicated by Nakasuji's $\Omega$-result~\cites[Theorem~1.2]{Nakasuji2000}[Theorem~1.1]{Nakasuji2001}, although~the supporting evidence is (even numerically) less abundant than in the $2$-dimensional scenario.

Let $\mathbb{H}_{3} \coloneqq \mathbb{C} \times \mathbb{R}_{+}^{\times}$ denote the quaternionic hyperbolic space. For $\omega \coloneqq e^{\frac{2\pi i}{3}}$, let $\mathbb{Q}(\omega)$ denote the Eisenstein quadratic field of class number $1$ with its ring of integers $\mathbb{Z}[\omega]$, discriminant~$-3$, and the unique ramified prime $\lambda \coloneqq \sqrt{-3} = 1+2\omega$. Let $\Gamma = \mathrm{PSL}_{2}(\mathbb{Z}[\omega])$, let $\Gamma_{1}(3)$ denote the principal congruence subgroup of level $3$, and let $\Gamma_{2} \coloneqq \langle \mathrm{PSL}_{2}(\mathbb{Z}), \Gamma_{1}(3) \rangle$. If $\gamma = \begin{psmallmatrix} a & b \\ c & d \end{psmallmatrix} \in \Gamma_{1}(3)$ and $(\frac{c}{\cdot})_{3}$ denotes the cubic symbol over $\mathbb{Z}[\omega]$, then the cubic Kubota character $\chi: \Gamma_{1}(3) \to \{1, \omega, \omega^{2} \}$ is defined by~\cites[Satz]{Kubota1966}
\begin{equation}\label{eq:Kubota-symbol}
\chi_{3}(\gamma) \coloneqq 
	\begin{dcases}
	\Big(\frac{c}{d} \Big)_{3} & \text{if $c \ne 0$},\\
	1 & \text{if $c = 0$},
	\end{dcases}
\end{equation}
which extends to a well-defined homomorphism via the law of cubic reciprocity. Reminiscent of the standard Eisenstein series with a simple pole at $s = 2$, the metaplectic Eisenstein series has a simple pole at $s = \frac{4}{3}$; see~\cites[Section~3]{Kubota1968} for a comprehensive account of metaplectic Eisenstein series in greater generality. Notably, if $\mathcal{L}(\Gamma, \chi, s)$ stands for the finite-dimensional Hilbert space of square-integrable automorphic forms under $\Gamma$ with character $\chi$ and spectral parameter $s$, then it corresponds to the \textit{cubic theta function} $\vartheta_{3} \in \mathcal{L}(\Gamma_{2}, \chi_{3}, \frac{4}{3})$, whose~Shimura correspondent is the constant eigenfunction $1 \in \mathcal{L}(\Gamma, \mathbbm{1}, 2)$.

As the main term in the prime geodesic theorem is dictated by the largest residual Laplace eigenvalue on the underlying orbifold, the \textit{metaplectic prime geodesic theorem} aims to identify the optimal rate of convergence in the asymptotic~formula
\begin{equation}\label{eq:metaplectic-PGT}
\Psi_{\Gamma_{2}}^{(3)}(x, \chi_{3}) = \frac{3}{4} x^{\frac{4}{3}}+\mathcal{E}_{\Gamma_{2}}^{(3)}(x, \chi_{3}).
\end{equation}
\begin{definition}\label{def:delta-k-nu-3}
Let $\chi_{3}$ denote the cubic Kubota character. Then we define $\delta_{1}^{(3)}(\chi_{3}) \in [1, \frac{5}{3}]$~so that
\begin{equation}
\mathcal{E}_{\Gamma_{2}}^{(3)}(x, \chi_{3}) \ll_{\varepsilon} x^{\delta_{1}^{(3)}(\chi_{3})+\varepsilon},
\end{equation}
while we define $\delta_{2}^{(3)}(\chi_{3}) \in [1, \frac{8}{5}]$ so that there exists a constant $\eta_{2}^{(3)}(\chi_{3}) \geq 0$ such that
\begin{equation}
\Big(\frac{1}{Y} \int_{X}^{X+Y} |\mathcal{E}_{\Gamma_{2}}^{(3)}(x, \chi_{3})|^{2} \, dx \Big)^{\frac{1}{2}} \ll_{\varepsilon} X^{\delta_{2}^{(3)}(\chi_{3})+\varepsilon} \Big(\frac{X}{Y} \Big)^{\eta_{2}^{(3)}(\chi_{3})}.
\end{equation}
\end{definition}

While the proofs are omitted to avoid superfluous methodological duplications, the~trivial bound in conjunction with \cref{thm:main,thm:CG} reads (cf.~\cref{prop:explicit-formula-metaplectic})
\begin{equation}\label{eq:barrier-metaplectic}
\delta_{1}^{(3)}(\chi_{3}) \leq \frac{11}{9}, \qquad \delta_{2}^{(3)}(\chi_{3}) \leq \frac{7}{5}, \qquad \eta_{2}^{(3)}(\chi_{3}) \leq \frac{1}{5}.
\end{equation}
The following result constitutes a refinement over~\eqref{eq:barrier-metaplectic}, contingent upon the unconditional resolution of the mean Lindel\"{o}f hypothesis over $\mathbb{Q}(\omega)$ akin to~\cites[Theorem~3.2]{Kaneko2025-2}. This, in principle, is the sole methodological justification for restricting our focus to the cubic~cover.
\begin{theorem}\label{thm:main-metaplectic}
Let $\chi_{3}$ denote the cubic Kubota character. Then unconditionally
\begin{equation}\label{eq:main-metaplectic-unconditional}
\delta_{1}^{(3)}(\chi_{3}) \leq \frac{25}{21} = 1.19047 \cdots.
\end{equation}
\end{theorem}

\begingroup
\renewcommand{\arraystretch}{1.5}
\begin{table}[ht!]
\centering
\setlength\tabcolsep{10pt}
\caption{A Shimura correspondence between $\delta_{1}^{(3)}(\mathbbm{1})$ and $\delta_{1}^{(3)}(\chi_{3})$}
\label{table:2}
\begin{tabular}{|c|c|c|c|}
	\hline
	Source & $\delta_{1}^{(3)}(\mathbbm{1})$ & $\delta_{1}^{(3)}(\chi_{3})$ & $\delta_{1}^{(3)}(\mathbbm{1})-3(\delta_{1}^{(3)}(\chi_{3})-1)$ \\ \hline
	Trivial & $2$ & $\frac{4}{3}$ & $1$ \\ \hline
	\cites{Sarnak1983} \& \eqref{eq:barrier-metaplectic} & $\frac{5}{3}$ & $\frac{11}{9}$ & $1$ \\ \hline
	\cites{Koyama2001} \& \eqref{eq:main-metaplectic-unconditional} & $\frac{11}{7}$ & $\frac{25}{21}$ & $1$ \\ \hline
\end{tabular}
\end{table}
\endgroup

Similarly, the generalised Shimura correspondence of Flicker~\cites[Theorem~5.3]{Flicker1980} provides a quantitative paring of $\delta_{1}^{(3)}(\mathbbm{1})$ and $\delta_{1}^{(3)}(\chi_{3})$, as summarised in \cref{table:2}. A reformulation in the classical language of the Shimura correspondence of Flicker ensures the transition between the spectral parameters attached to $f \in \mathcal{L}(\Gamma, \mathbbm{1})$ and $\tilde{f} \in \mathcal{L}(\Gamma_{2}, \chi_{3})$, namely $s_{f}-1 = \pm 3(s_{\tilde{f}}-1)$; cf. \cref{thm:Shimura}. It deserves further penetration to demonstrate as a possible straightforward generalisation that if $\chi_{n}$ denotes the $n$-fold Kubota character, then (cf.~\cites[Theorem~1.1]{Koyama2001})
\begin{equation}
\delta_{1}^{(3)}(\chi_{n}) \leq 1+\frac{4}{7n}.
\end{equation}
We refrain from undertaking such additional work.

\section{Scalar-valued theory}
This section compiles basic facts on automorphic forms of arbitrary real weight, which~are considered fundamental in the higher echelon; see~\cites{AhlgrenAndersen2018}{AndersenDuke2020}{DukeFriedlanderIwaniec2012}{Pribitkin2000}{Proskurin2003}{Sarnak1984}{Stromberg2008} for a comprehensive summary of the theoretical background. The overall presentation is structured to remain as self-contained as possible to accommodate a diverse readership.

\subsection{Multiplier systems}\label{subsect:multiplier-systems}
Let $k \in \mathbb{R}$. Given $\gamma = \begin{psmallmatrix} a & b \\ c & d \end{psmallmatrix} \in \mathrm{SL}_{2}(\mathbb{R})$ and $z \in \mathbb{H}_{2}$, we define
\begin{equation}\label{eq:i-j}
i(\gamma, z) \coloneqq cz+d, \qquad j(\gamma, z) \coloneqq \frac{cz+d}{|cz+d|} = e^{i \arg(cz+d)}.
\end{equation}
As the cocycle relation $i(\gamma_{1} \gamma_{2}, z) = i(\gamma_{1}, \gamma_{2} z) i(\gamma_{2}, z)$ holds for all $\gamma_{1}, \gamma_{2} \in \mathrm{SL}_{2}(\mathbb{R})$, the quantity
\begin{equation}
\tilde{\omega}(\gamma_{1}, \gamma_{2}) \coloneqq \frac{1}{2\pi}(\arg i(\gamma_{1}, \gamma_{2} z)+\arg i(\gamma_{2}, z)-\arg i(\gamma_{1} \gamma_{2}, z))
\end{equation}
is an integer independent of $z$. The factor system of weight $k$ is then defined by
\begin{equation}
\omega(\gamma_{1}, \gamma_{2}) \coloneqq e(k \, \tilde{\omega}(\gamma_{1}, \gamma_{2})) = j(\gamma_{2}, z)^{k} j(\gamma_{1}, \gamma_{2} z)^{k} j(\gamma_{1} \gamma_{2}, z)^{-k}, \qquad 
\gamma_{1}, \gamma_{2} \in \mathrm{SL}_{2}(\mathbb{R}).
\end{equation}
A multiplier system of weight $k$ and dimension $1$ on a cofinite Fuchsian group $\Gamma \subset \mathrm{SL}_{2}(\mathbb{R})$~is a map $\nu: \Gamma \to \mathbb{C}^{\times}$ satisfying the the conditions
\begin{itemize}
\item $|\nu(\gamma)| = 1$;
\item $\nu(-I) = e^{-k \pi i}$;
\item $\nu(\gamma_{1} \gamma_{2}) = \omega(\gamma_{1}, \gamma_{2}) \nu(\gamma_{1}) \nu(\gamma_{2})$ for all $\gamma_{1}, \gamma_{2} \in \Gamma$.
\end{itemize}

If $\nu$ is a multiplier system of weight $k$, then it is a multiplier system of weight $k^{\prime}$ for every $k^{\prime} \equiv k \tpmod{2}$, and its conjugate $\overline{\nu}$ is a multiplier system of weight $-k$. In anticipation of subsequent discussions, it is convenient to define the slash operator of weight $k$ by
\begin{equation}\label{eq:slash}
(f|_{k} \gamma)(z) \coloneqq j(\gamma, z)^{-k} f(\gamma z).
\end{equation}

\subsection{Kloosterman sums}\label{subsect:Kloosterman-sums}
Fix a multiplier system $\nu$ of weight $k$ on $\Gamma \subset \operatorname{SL}_{2}(\mathbb{R})$, which acts transitively on $\mathbb{P}^{1}(\mathbb{Q})$ via M\"{o}bius transformations. An element $\mathfrak{a} \in \mathbb{P}^{1}(\mathbb{Q})$ is called a cusp, and two cusps $\mathfrak{a}$ and $\mathfrak{b}$ are called equivalent under $\Gamma$ if there exists $\gamma \in \Gamma$ satisfying $\mathfrak{a} = \gamma \mathfrak{b}$. Let $\Gamma_{\mathfrak{a}} \coloneqq \langle \pm \gamma_{\mathfrak{a}} \rangle$ denote the stabiliser of the cusp $\mathfrak{a}$ in $\Gamma$. For example, $\Gamma_{\infty} = \{\pm \begin{psmallmatrix} 1 & b \\ 0 & 1 \end{psmallmatrix}: b \in \mathbb{Z} \}$. Let $\sigma_{\mathfrak{a}} \in \mathrm{SL}_{2}(\mathbb{R})$ denote the scaling matrix such that $\sigma_{\mathfrak{a}} \infty = \mathfrak{a}$ and $\sigma_{\mathfrak{a}}^{-1} \Gamma_{\mathfrak{a}} \sigma_{\mathfrak{a}} = \Gamma_{\infty}$, determined up to composition with a translation from the right side. For any cusp $\mathfrak{a}$, let $\kappa_{\mathfrak{a}} \in [0, 1)$~satisfy
\begin{equation}\label{eq:kappa}
\nu(\gamma_{\mathfrak{a}}) = e(\kappa_{\mathfrak{a}}).
\end{equation}
The cusp $\mathfrak{a}$ is called essential with respect to $\nu$ if the restriction of $\nu$ to $\Gamma_{\mathfrak{a}}$ is trivial, namely if $\kappa_{\mathfrak{a}} = 0$, and is called regular otherwise. For $\mathfrak{a} = \infty$, one may suppress the subscript unless it induces confusion, and write $\kappa \coloneqq \kappa_{\mathfrak{a}}$. For notational simplicity, we abbreviate $n_{\mathfrak{a}} \coloneqq n+\kappa_{\mathfrak{a}}$ for any $n \in \mathbb{Z}$. Let $\nu_{\mathfrak{ab}}$ denote the multiplier system for the conjugate group $\sigma_{\mathfrak{a}}^{-1} \Gamma \sigma_{\mathfrak{b}}$ given~by
\begin{equation}
\nu_{\mathfrak{ab}}(\gamma) \coloneqq \nu(\sigma_{\mathfrak{a}} \gamma \sigma_{\mathfrak{b}}^{-1}) \omega(\sigma_{\mathfrak{a}}^{-1}, \sigma_{\mathfrak{a}} \gamma \sigma_{\mathfrak{b}}^{-1}) \omega(\gamma \sigma_{\mathfrak{b}}^{-1}, \sigma_{\mathfrak{b}}).
\end{equation}
Now, if the set of allowed moduli is denoted by $\mathcal{C}(\mathfrak{a}, \mathfrak{b}) \coloneqq \{c > 0: 
\begin{psmallmatrix} \ast & \ast \\ c & \ast \end{psmallmatrix} \in \sigma_{\mathfrak{a}}^{-1} \Gamma \sigma_{\mathfrak{b}} \}$, then the Kloosterman sums attached to the pair of cusps $(\mathfrak{a}, \mathfrak{b})$ with respect to the multiplier system $\nu$ are defined by~\cites[Page~700]{Hejhal1983}
\begin{equation}\label{eq:Kloosterman}
S_{\mathfrak{ab}}(m, n, c, \nu) \coloneqq \sum_{\gamma = \begin{psmallmatrix} a & \ast \\ c & d \end{psmallmatrix} \in \Gamma_{\infty} \backslash \sigma_{\mathfrak{a}}^{-1} \Gamma \sigma_{\mathfrak{b}}/\Gamma_{\infty}} \overline{\nu_{\mathfrak{ab}}(\gamma)} e \Big(\frac{m_{\mathfrak{a}} a+n_{\mathfrak{b}} d}{c} \Big), \qquad c \in \mathcal{C}(\mathfrak{a}, \mathfrak{b}).
\end{equation}

\subsection{Laplace eigenfunctions}\label{subsect:Laplace-Beltrami}
A function $f: \mathbb{H}_{2} \to \mathbb{C}$ is said to be an automorphic form of weight $k$ and multiplier system $\nu$ on $\Gamma$ if it transforms for all $\gamma \in \Gamma$ as
\begin{equation}\label{eq:automorphy}
(f|_{k} \gamma)(z) = \nu(\gamma) f(z).
\end{equation}
Let $\mathcal{A}_{k}(\Gamma, \nu)$ denote the linear space of all such functions, and let $\mathcal{L}_{k}(\Gamma, \nu) \subset \mathcal{A}_{k}(\Gamma, \nu)$ denote the subspace consisting of square-integrable functions on the fundamental domain $\Gamma \backslash \mathbb{H}_{2}$ with respect to the $\mathrm{SL}_{2}(\mathbb{R})$-invariant measure $d\mu(z) \coloneqq y^{-2} \, dx \, dy$ and the Petersson inner product
\begin{equation}
\langle f, g \rangle \coloneqq \int_{\Gamma \backslash \mathbb{H}_{2}} f(z) \overline{g(z)} \, d\mu(z)
\end{equation}
for $f, g \in \mathcal{L}_{k}(\Gamma, \nu)$. For $k \in \mathbb{R}$, the hyperbolic Laplacian of weight $k$ is defined by
\begin{equation}\label{eq:Laplace-Beltrami}
\Delta_{k} \coloneqq y^{2} \Big(\frac{\partial^{2}}{\partial^{2} x}+\frac{\partial^{2}}{\partial^{2} y} \Big)-iky \frac{\partial}{\partial x}.
\end{equation}
It admits a unique self-adjoint extension to $\mathcal{L}_{k}(\Gamma, \nu)$, which we also denote by $\Delta_{k}$ by abuse~of notation. For each $k \in \mathbb{R}$, $\Delta_{k}$ commutes with the slash operator~\eqref{eq:slash} for all $\gamma \in \mathrm{SL}_{2}(\mathbb{R})$.

\subsection{Maa{\ss} cusp forms}
A real analytic function $f: \mathbb{H}_{2} \to \mathbb{C}$ is said to be an eigenfunction of $\Delta_{k}$ with Laplace eigenvalue $\lambda_{f} \in \mathbb{C}$ if
\begin{equation}
\Delta_{k} f = -\lambda_{f} f.
\end{equation}
An eigenfunction $f$ is said to be a Maa{\ss} form if $f \in \mathcal{A}_{k}(\Gamma, \nu)$ is smooth and obeys the growth condition
\begin{equation}
(f|_{k} \gamma)(z) \ll y^{\sigma}+y^{1-\sigma}
\end{equation}
for all $\gamma \in \Gamma$, $z \in \mathbb{H}_{2}$, and for some $\sigma$ dependent on $\gamma$. If a Maa{\ss} form $f$ obeys the additional cuspidality condition
\begin{equation}
\int_{0}^{1} (f|_{k} \sigma_{\mathfrak{a}})(z) e(\kappa_{\mathfrak{a}} x) \, dx = 0
\end{equation}
for every cusp $\mathfrak{a}$ of $\Gamma$, then $f \in \mathcal{L}_{k}(\Gamma, \nu)$, and $f$ is said to be a Maa{\ss} cusp form. 

Let $\mathcal{B}_{k}(\Gamma, \nu) \subset \mathcal{A}_{k}(\Gamma, \nu)$ denote the space of smooth functions $f$ such that $f$ and $\Delta_{k} f$ are both bounded. It follows that $\mathcal{B}_{k}(\Gamma, \nu) \subset \mathcal{L}_{k}(\Gamma, \nu)$, $\mathcal{B}_{k}(\Gamma, \nu)$ is dense in $\mathcal{L}_{k}(\Gamma, \nu)$, and $\Delta_{k}$ is a symmetric operator on $\mathcal{B}_{k}(\Gamma, \nu)$. If the bottom eigenvalue is denoted by $\lambda_{0} \coloneqq \frac{|k|}{2}(1-\frac{|k|}{2})$, then $\langle f, -\Delta_{k} f \rangle \geq \lambda_{0} \langle f, f \rangle$ for any $f \in \mathcal{L}_{k}(\Gamma, \nu)$. By a theorem of Friedrichs, $-\Delta_{k}$ admits a unique self-adjoint extension to $\mathcal{L}_{k}(\Gamma, \nu)$, and by a theorem of von Neumann, $\mathcal{L}_{k}(\Gamma, \nu)$ has a complete spectral resolution with respect to $-\Delta_{k}$. The spectrum consists of two~distinct components: the continuous spectrum in $[\frac{1}{4}, \infty)$ arising from Eisenstein series $E_{\mathfrak{a}}(z, s, \nu)$ for~each essential cusp $\mathfrak{a}$, and the discrete spectrum of finite multiplicity in $[\lambda_{0}, \infty)$. A portion of the discrete spectrum arises from residues of the Eisenstein series at possible simple poles in $(\frac{1}{2}, 1-\frac{|k|}{2}]$, and the remainder of the discrete spectrum arises from Maa{\ss} cusp forms.

For each Laplace eigenvalue $\lambda_{f}$, we employ the standard conventions
\begin{equation}\label{eq:convention-eigenvalues}
\lambda_{f} = \tfrac{1}{4}+t_{f}^{2} = s_{f}(1-s_{f}), \qquad s_{f} = \tfrac{1}{2}+it_{f}, \qquad t_{f} \in i(0, \tfrac{1-|k|}{2}] \cup [0, \infty).
\end{equation}
If $\mathcal{L}_{k}(\Gamma, \nu, s_{f}) \subset \mathcal{L}_{k}(\Gamma, \nu)$ denotes the subspace associated to the spectral parameter\footnote{The quantity $t_{f}$ may also be referred to as the spectral parameter of $f$ as the situation demands.} $s_{f}$, then complex conjugation yields an isometry $\mathcal{L}_{k}(\Gamma, \nu, s_{f}) \leftrightarrow \mathcal{L}^{\ast}_{-k}(\Gamma, \overline{\nu}, s_{f})$ between normed spaces. Each $f \in \mathcal{L}_{k}(\Gamma, \nu, s_{f})$ admits a Fourier--Whittaker expansion around $\mathfrak{a}$ of the shape
\begin{equation}\label{eq:Fourier-Whittaker}
j(\sigma_{\mathfrak{a}}, z)^{-k} f(\sigma_{\mathfrak{a}} z) = c_{0}(y)+\sum_{n_{\mathfrak{a}} \ne 0} \rho_{f \mathfrak{a}}(n) W_{\frac{k}{2} \, \mathrm{sgn}(n_{\mathfrak{a}}), it_{f}}(4\pi |n_{\mathfrak{a}}|y) e(n_{\mathfrak{a}} x),
\end{equation}
where $W_{\kappa, \mu}(\cdot)$ denotes the standard Whittaker function, and
\begin{equation}\label{eq:c-0}
c_{0}(y) \coloneqq 
	\begin{cases}
	0 & \text{if $\kappa_{\mathfrak{a}} \ne 0$},\\
	0 & \text{if $\kappa_{\mathfrak{a}} = 0$ and $t_{f} \geq 0$},\\
	\rho_{f \mathfrak{a}}(0) y^{\frac{1}{2}+it}+\rho_{f \mathfrak{a}}^{\prime}(0) y^{\frac{1}{2}-it} & \text{if $\kappa_{\mathfrak{a}} = 0$ and $t_{f} \in i(0, \tfrac{1-|k|}{2}]$}.
	\end{cases}
\end{equation}
In the third case, $\rho_{f \mathfrak{a}}(0) \ne 0$ if and only if $f$ arises from a residue of an Eisenstein series.

\subsection{Holomorphic cusp forms}\label{subsect:holomorphic}
Let $\mathcal{M}_{k}(\Gamma, \nu)$ stand for the space of holomorphic modular forms of weight $k$ and multiplier system $\nu$. If $f \in \mathcal{L}_{k}(\Gamma, \nu)$ has the bottom Laplace eigenvalue $\lambda_{0}$, then $f(z)$ lies in the kernel of $L_{k}$ when $k \geq 0$, whereas $\overline{f(z)}$ lies in the kernel of $L_{-k}$ when $k < 0$. Automorphy~\eqref{eq:automorphy} shows that the normalised function
\begin{equation}
F(z) \coloneqq 
	\begin{cases}
	y^{-\frac{k}{2}} f(z) & \text{if $k \geq 0$},\\
	y^{\frac{k}{2}} \overline{f(z)} & \text{if $k < 0$}
	\end{cases}
\end{equation}
lies in $\mathcal{M}_{k}(\Gamma, \nu)$ when $k \geq 0$ and lies in $\mathcal{M}_{-k}(\Gamma, \overline{\nu})$ when $k < 0$. Hence, there exists a bijective correspondence between all $f \in \mathcal{L}_{k}(\Gamma, \nu)$ with eigenvalue $\lambda_{0}$ and holomorphic modular forms $F$ of weight $k$, and $F$ is a cusp form if and only if $f$ is a Maa{\ss} cusp form.

\subsection{Eisenstein series}
For $k \in (-1, 1]$, fix an essential cusp $\mathfrak{a}$ with respect to the multiplier system $\nu$ of weight $k$ on $\Gamma$. The Eisenstein series associated to $\mathfrak{a}$ is defined by
\begin{equation}\label{eq:Eisenstein-def}
E_{\mathfrak{a}}(z, s, \nu) \coloneqq \sum_{\gamma \in \Gamma_{\mathfrak{a}} \backslash \Gamma} \overline{\nu(\gamma) \omega(\sigma_{\mathfrak{a}}^{-1}, \gamma)} \Im(\sigma_{\mathfrak{a}}^{-1} \gamma z)^{s} j(\sigma_{\mathfrak{a}}^{-1} \gamma, z)^{-k},
\end{equation}
which converges absolutely and uniformly on any compact subset of $\Gamma \backslash \mathbb{H}_{2}$ as long as $\Re(s) > 1$, extending meromorphically to $\mathbb{C}$ but not being square-integrable over $\Gamma \backslash \mathbb{H}_{2}$. For any cusp $\mathfrak{b}$ of $\Gamma$, it admits a Fourier--Whittaker expansion around $\mathfrak{b}$ of the shape
\begin{multline}\label{eq:Eisenstein}
j(\sigma_{\mathfrak{b}}, z)^{-k} E_{\mathfrak{a}}(\sigma_{\mathfrak{b}} z, s, \nu)
 = \delta_{\mathfrak{a} = \mathfrak{b}} \, y^{s}+\delta_{\kappa_{\mathfrak{b}} = 0} \, \rho_{\mathfrak{ab}}(0, s, \nu) y^{1-s}\\
 + \sum_{n_{\mathfrak{b}} \ne 0} \rho_{\mathfrak{ab}}(n, s, \nu) W_{\frac{k}{2} \, \mathrm{sgn}(n_{\mathfrak{b}}), s-\frac{1}{2}}(4\pi |n_{\mathfrak{b}}|y) e(n_{\mathfrak{b}} x),
\end{multline}
where
\begin{equation}
\rho_{\mathfrak{ab}}(n, s, \nu) \coloneqq 
	\begin{dcases}
	\frac{e(-\frac{k}{4}) \pi^{s} |n_{\mathfrak{b}}|^{s-1}}{\Gamma(s+\frac{k}{2} \, \mathrm{sgn}(n_{\mathfrak{b}}))} \sum_{c \in \mathcal{C}(\mathfrak{a}, \mathfrak{b})} \frac{S_{\mathfrak{ab}}(0, n, c, \nu)}{c^{2s}} & \text{if $n_{\mathfrak{b}} \ne 0$},\\
	\frac{e(-\frac{k}{4}) 4^{1-s} \pi \Gamma(2s-1)}{\Gamma(s+\frac{k}{2}) \Gamma(s-\frac{k}{2})} \sum_{c \in \mathcal{C}(\mathfrak{a}, \mathfrak{b})} \frac{S_{\mathfrak{ab}}(0, 0, c, \nu)}{c^{2s}} & \text{if $n_{\mathfrak{b}} = 0$}.
	\end{dcases}
\end{equation}
The Fourier--Whittaker coefficients $\rho_{\mathfrak{ab}}(n, s, \nu)$ are continued meromorphically to $\mathbb{C}$ and well-defined on $\Re(s) = \frac{1}{2}$. The entries of the scattering matrix are indexed by pairs of essential cusps, namely $\Phi(s, \nu) \coloneqq (\rho_{\mathfrak{a} \mathfrak{b}}(0, s, \nu))_{\mathfrak{a}, \mathfrak{b}}$. If $h$ stands for the number of inequivalent essential cusps of $\Gamma$, then the $h \times 1$ vector $\mathcal{E}(z, s, \nu)$ of all Eisenstein series at these cusps satisfies the functional equation
\begin{equation}
\mathcal{E}(z, 1-s, \nu) = \Phi(s, \nu) \mathcal{E}(z, s, \nu),
\end{equation}
from which it follows that
\begin{equation}\label{eq:s-1-s}
\Phi(s, \nu) \Phi(1-s, \nu) = I_{h}.
\end{equation}
The scattering determinant $\varphi \coloneqq \det \Phi$ possesses a Dirichlet series expansion of the shape
\begin{equation}
\varphi(s, \nu) \coloneqq \Big(\frac{\sqrt{\pi} 4^{1-s} \Gamma(2s-1)}{\Gamma(s+\frac{k}{2}) \Gamma(s-\frac{k}{2})} \Big)^{h} \sum_{n = 1}^{\infty} \frac{a_{n}}{b_{n}^{2s}},
\end{equation}
where $(b_{n})_{n \in \mathbb{N}}$ is a strictly increasing sequence of positive coefficients, and the series converges absolutely for $\Re(s) > 1$ and is holomorphic for $\Re(s) \geq \frac{1}{2}$ except possibly for a finite number of poles corresponding to the residual spectrum of $\Delta_{k}$.

\section{Vector-valued theory}\label{sect:vector-valued-theory}
Given a vector or a matrix $M$, let $M^{\mathrm{T}}$ denote its transpose, and let $M^{\mathrm{H}}$ denote its conjugate transpose. Let $\vec{u}$ and $\vec{F}(z)$ denote a complex vector and a vector-valued function of dimension $D \geq 1$, respectively. Given $1 \leq \ell \leq D$, let $\mathfrak{e}_{\ell} \coloneqq (0, \ldots, 0, 1, 0, \ldots, 0)^{\mathrm{T}}$ denote the unit vector having $1$ at its $\ell$-th entry and $0$ otherwise. When the superscript $(\ell)$ is present, we write
\begin{equation}
\vec{u} \coloneqq \sum_{\ell = 1}^{D} \vec{u}^{\, (\ell)} \coloneqq (u^{(1)}, u^{(2)}, \ldots, u^{(D)})^{\mathrm{T}}, \qquad \vec{F}(z) \coloneqq \sum_{\ell = 1}^{D} \vec{F}^{(\ell)}(z) \coloneqq \sum_{\ell = 1}^{D} F^{(\ell)}(z) \mathfrak{e}_{\ell}.
\end{equation}

\subsection{Vector-valued multiplier systems}
Following~\cites{KnoppMason2003}{KnoppMason2004}, we are now prepared to define vector-valued multiplier systems on any cofinite Fuchsian group $\Gamma \subset \mathrm{SL}_{2}(\mathbb{R})$, where $\begin{psmallmatrix} 1 & 1 \\ 0 & 1 \end{psmallmatrix} \in \Gamma$ is assumed without loss of generality. It is convenient to fix the principal argument in $(-\pi, \pi]$ and employ the factor of automorphy $j(\gamma, z)$ as in~\eqref{eq:i-j}. Given $\gamma = \begin{psmallmatrix} a & b \\ c & d \end{psmallmatrix} \in \mathrm{SL}_{2}(\mathbb{R})$ and $z \in \mathbb{H}_{2}$, we define the vector-valued slash operator of weight $k$ by
\begin{equation}
(\vec{F}|_{k} \gamma)(z) = ((F^{(1)}|_{k} \gamma)(z), (F^{(2)}|_{k} \gamma)(z), \ldots, (F^{(D)}|_{k} \gamma)(z))^{\mathrm{T}} \coloneqq j(\gamma, z)^{-k} \vec{F}(\gamma z).
\end{equation}
A multiplier system of weight $k$ and dimension $D$ on a cofinite Fuchsian group $\Gamma \subset \mathrm{SL}_{2}(\mathbb{R})$~is a map $\xi: \Gamma \to \mathrm{GL}_{D}(\mathbb{C})$ satisfying the conditions
\begin{itemize}
\item $\xi(\gamma)$ is unitary for all $\gamma \in \Gamma$, namely $\xi(\gamma)^{-1} = \xi(\gamma)^{\mathrm{H}}$;
\item $\xi(-I) = e^{-k\pi i} I_{D}$, where $I_{D} \in \mathrm{GL}_{D}(\mathbb{C})$ denotes the identity matrix;
\item $\xi(\gamma_{1} \gamma_{2}) = \omega(\gamma_{1}, \gamma_{2}) \xi(\gamma_{1}) \xi(\gamma_{2})$ for all $\gamma_{1}, \gamma_{2} \in \Gamma$.
\end{itemize}

\subsection{Vector-valued Kloosterman sums}
A multiplier system of weight $k$ and dimension $D$ on $\Gamma \subset \operatorname{SL}_{2}(\mathbb{R})$ acts transitively on $\mathbb{P}^{1}(\mathbb{Q})$ by M\"{o}bius transformations. In conjunction~with the scalar-valued setting, for every cusp $\mathfrak{a}$ of $\Gamma$, we define the quantity $\kappa_{\mathfrak{a}}^{(\ell)} \in [0,1)$ by
\begin{equation}\label{eq:kappa-2}
\xi(\gamma_{\mathfrak{a}}) = \mathrm{diag} \{e(\kappa_{\mathfrak{a}}^{(1)}), \ldots, e(\kappa_{\mathfrak{a}}^{(D)}) \}.
\end{equation}
The cusp $\mathfrak{a}$ is called essential with respect to $\xi$ if the restriction of $\xi$ to $\Gamma_{\mathfrak{a}}$ is trivial, namely if $\kappa_{\mathfrak{a}}^{(\ell)} = 0$ for every $1 \leq \ell \leq D$, and is called regular otherwise. One may write $n_{\mathfrak{a}}^{(\ell)} \coloneqq n+\kappa_{\mathfrak{a}}^{(\ell)}$ for any $n \in \mathbb{Z}$. Let $\xi_{\mathfrak{ab}}$ denote the multiplier system for the conjugate group $\sigma_{\mathfrak{a}}^{-1} \Gamma \sigma_{\mathfrak{b}}$ given~by
\begin{equation}
\xi_{\mathfrak{ab}}(\gamma) \coloneqq \xi(\sigma_{\mathfrak{a}} \gamma \sigma_{\mathfrak{b}}^{-1}) \omega(\sigma_{\mathfrak{a}}^{-1}, \sigma_{\mathfrak{a}} \gamma \sigma_{\mathfrak{b}}^{-1}) \omega(\gamma \sigma_{\mathfrak{b}}^{-1}, \sigma_{\mathfrak{b}}).
\end{equation}
Furthermore, if $\xi(\gamma_{\mathfrak{a}})$ has eigenvalues $e(\kappa_{\mathfrak{a}}^{(1)}), \ldots, e(\kappa_{\mathfrak{a}}^{(D)})$, then the corresponding orthonormal set of eigenvectors is denoted by $\{\vec{\mathfrak{f}}_{\mathfrak{a}}^{\, (\ell)} \}$. For $c \in \mathcal{C}(\mathfrak{a}, \mathfrak{b})$, the vector-valued Kloosterman sums attached to the pair of cusps $(\mathfrak{a}, \mathfrak{b})$ with respect to $\xi$ are defined by~\cites[Equation~(3)]{Matthes1994-2}\footnote{There is a misprint in the numbering of~\cites[Equation~(3)]{Matthes1994-2}, which is incorrectly written as (13).}
\begin{equation}
S_{\mathfrak{ab}}^{(\ell)}(m, n, c, \xi) \coloneqq \sum_{\gamma = \begin{psmallmatrix} a & \ast \\ c & d \end{psmallmatrix} \in \Gamma_{\infty} \backslash \sigma_{\mathfrak{a}}^{-1} \Gamma \sigma_{\mathfrak{b}}/\Gamma_{\infty}} (\vec{\mathfrak{f}}_{\mathfrak{a}}^{\, (\ell)})^{\mathrm{H}} \, \overline{\xi_{\mathfrak{ab}}(\gamma)} \, \vec{\mathfrak{f}}_{\mathfrak{b}}^{\, (\ell)} e \Big(\frac{m_{\mathfrak{a}}^{(\ell)} a+n_{\mathfrak{b}}^{(\ell)} d}{c} \Big).
\end{equation}
If $\xi$ is induced by a $1$-dimensional multiplier system, then it is known~\cites[Equation~(5)]{Matthes1994-2} that the vector-valued Kloosterman sums admit more explicit realisations.

\subsection{Vector-valued automorphic forms}
A function $\vec{F}: \mathbb{H}_{2} \to \mathbb{C}^{D}$ is said to be a vector-valued automorphic form of weight $k$ and multiplier system $\xi$ of dimension $D$ on $\Gamma$ if it transforms for all $\gamma \in \Gamma$ as
\begin{equation}
(\vec{F}|_{k} \gamma)(z) = \xi(\gamma) \vec{F}(z).
\end{equation}
Let $\mathcal{A}_{k}(\Gamma, \xi)$ denote the linear space of all such automorphic forms. If $\vec{F}, \vec{G} \in \mathcal{A}_{k}(\Gamma, \xi)$, then their Petersson inner product is defined formally by
\begin{equation}\label{eq:Petersson}
\langle \vec{F}, \vec{G} \rangle \coloneqq \int_{\Gamma \backslash \mathbb{H}_{2}} \sum_{\ell = 1}^{D} F^{(\ell)}(z) \overline{G^{(\ell)}(z)} \, d\mu(z) = \int_{\Gamma \backslash \mathbb{H}_{2}} \vec{G}^{\, \mathrm{H}}(z) \vec{F}(z) \, d\mu(z).
\end{equation}
Let $\mathcal{L}_{k}(\Gamma, \xi)\subset \mathcal{A}_{k}(\Gamma, \xi)$ denote the subspace consisting of square-integrable functions on the fundamental domain $\Gamma \backslash \mathbb{H}_{2}$ with respect to the Petersson inner product~\eqref{eq:Petersson}.

A real analytic function $\vec{F}: \mathbb{H}_{2} \to \mathbb{C}^{D}$ such that each component $F^{(\ell)}$ is smooth is said to be a $D$-dimensional eigenfunction of $\Delta_{k}$ with Laplace eigenvalue $\lambda_{F} \in \mathbb{C}$ if
\begin{equation}
\Delta_{k} \vec{F}(z) = -\lambda_{F} \vec{F}(z).
\end{equation}
An eigenfunction $\vec{F} = (F^{(1)}, \ldots, F^{(D)})$ is said to be a vector-valued Maa{\ss} form of weight~$k$~and dimension $D$ if $F^{(\ell)} \in \mathcal{A}_{k}(\Gamma, \xi)$ is smooth and it obeys the growth condition
\begin{equation}
(F^{(\ell)}|_{k} \gamma)(z) \ll y^{\sigma}+y^{1-\sigma}
\end{equation}
for all $\gamma \in \Gamma$, $z \in \mathbb{H}_{2}$, and for some $\sigma$ dependent on $\gamma$. If a vector-valued Maa{\ss} form $\vec{F}$ obeys the additional cuspidality condition
\begin{equation}
\int_{0}^{1} (F^{(\ell)}|_{k} \sigma_{\mathfrak{a}})(z) e(\kappa_{\mathfrak{a}}^{(\ell)} x) \, dx = 0
\end{equation}
for every cusp $\mathfrak{a}$ of $\Gamma$ and $1 \leq \ell \leq D$, then $\vec{F} \in \mathcal{L}_{k}(\Gamma, \xi)$, and $\vec{F}$ is said to be a vector-valued Maa{\ss} cusp form of weight $k$ and dimension $D$. In conjunction with the scalar-valued setting, each $F^{(\ell)} \in \mathcal{L}_{k}(\Gamma, \xi, s_{F})$ admits a Fourier--Whittaker expansion around $\mathfrak{a}$ of the shape
\begin{align}
\begin{split}
j(\sigma_{\mathfrak{a}}, z)^{-k} F^{(\ell)}(\sigma_{\mathfrak{a}} z) = c_{0}^{(\ell)}(y)+\sum_{n_{\mathfrak{a}}^{(\ell)} \neq 0} \rho_{F \mathfrak{a}}^{(\ell)}(n) W_{\frac{k}{2} \, \mathrm{sgn}(n_{\mathfrak{a}}^{(\ell)}), it_{F}}(4\pi |n_{\mathfrak{a}}^{(\ell)}|y) e(n_{\mathfrak{a}}^{(\ell)} x),
\end{split}
\end{align}
where
\begin{equation}
c_{0}^{(\ell)}(y) \coloneqq 
	\begin{cases}
	0 & \text{if $\kappa_{\mathfrak{a}}^{(\ell)} \ne 0$},\\
	0 & \text{if $\kappa_{\mathfrak{a}}^{(\ell)} = 0$ and $t_{F} \geq 0$},\\
	\rho_{F \mathfrak{a}}^{(\ell)}(0) y^{\frac{1}{2}+it}+(\rho_{F \mathfrak{a}}^{(\ell)})^{\prime}(0) y^{\frac{1}{2}-it} & \text{if $\kappa_{\mathfrak{a}}^{(\ell)} = 0$ and $t_{F} \in i(0, \tfrac{1-|k|}{2}]$}.
	\end{cases}
\end{equation}

We shall refrain from explicating the companion vector-valued theory of holomorphic cusp forms to avoid undue redundancy with the scalar-valued theory; cf.~\cites[Section~4.4]{Sun2024-2}.

For $k \in (-1, 1]$, let $\mathfrak{a}$ be an essential cusp with respect to the multiplier system $\nu$ of weight $k$ on $\Gamma$. The vector-valued Eisenstein series associated to $\mathfrak{a}$ is defined by
\begin{equation}
\vec{E}_{\mathfrak{a}}(z, s, \nu) \coloneqq \sum_{\gamma \in \Gamma_{\mathfrak{a}} \backslash \Gamma} \overline{\xi(\gamma) \omega(\sigma_{\mathfrak{a}}^{-1}, \gamma)} \Im(\sigma_{\mathfrak{a}}^{-1} \gamma z)^{s} j(\sigma_{\mathfrak{a}}^{-1} \gamma, z)^{-k},
\end{equation}
which converges absolutely and uniformly on any compact subset of $\Gamma \backslash \mathbb{H}_{2}$ as long as $\Re(s) > 1$, extending meromorphically to $\mathbb{C}$ but not being square-integrable over $\Gamma \backslash \mathbb{H}_{2}$. For any cusp $\mathfrak{b}$ of $\Gamma$, it admits a Fourier--Whittaker expansion around $\mathfrak{b}$ of the shape
\begin{multline}
j(\sigma_{\mathfrak{b}}, z)^{-k} E_{\mathfrak{a}}^{(\ell)}(\sigma_{\mathfrak{b}} z, s, \nu)
 = \delta_{\mathfrak{a} = \mathfrak{b}} \, y^{s}+\delta_{\kappa_{\mathfrak{b}}^{(\ell)} = 0} \, \rho_{\mathfrak{ab}}^{(\ell)}(0, s, \nu) y^{1-s}\\
 + \sum_{n_{\mathfrak{b}} \ne 0} \rho_{\mathfrak{ab}}^{(\ell)}(n, s, \nu) W_{\frac{k}{2} \, \mathrm{sgn}(n_{\mathfrak{b}}^{(\ell)}), s-\frac{1}{2}}(4\pi |n_{\mathfrak{b}}^{(\ell)}|y) e(n_{\mathfrak{b}}^{(\ell)} x),
\end{multline}
where
\begin{equation}
\rho_{\mathfrak{ab}}^{(\ell)}(n, s, \nu) \coloneqq 
	\begin{dcases}
	\frac{e(-\frac{k}{4}) \pi^{s} |n_{\mathfrak{b}}^{(\ell)}|^{s-1}}{\Gamma(s+\frac{k}{2} \, \mathrm{sgn}(n_{\mathfrak{b}}^{(\ell)}))} \sum_{c \in \mathcal{C}(\mathfrak{a}, \mathfrak{b})} \frac{S_{\mathfrak{ab}}^{(\ell)}(0, n, c, \nu)}{c^{2s}} & \text{if $n_{\mathfrak{b}}^{(\ell)} \ne 0$},\\
	\frac{e(-\frac{k}{4}) 4^{1-s} \pi \Gamma(2s-1)}{\Gamma(s+\frac{k}{2}) \Gamma(s-\frac{k}{2})} \sum_{c \in \mathcal{C}(\mathfrak{a}, \mathfrak{b})} \frac{S_{\mathfrak{ab}}^{(\ell)}(0, 0, c, \nu)}{c^{2s}} & \text{if $n_{\mathfrak{b}}^{(\ell)} = 0$}.
	\end{dcases}
\end{equation}
The Fourier--Whittaker coefficients $\rho_{\mathfrak{ab}}^{(\ell)}(n, s, \nu)$ may be continued meromorphically to $\mathbb{C}$ and are well-defined on the unitary line $\Re(s) = \frac{1}{2}$.

\subsection{The vector-valued Selberg trace formula}\label{sect:Selberg-trace-formula}
Let $h: \mathbb{C} \to \mathbb{C}$ be a holomorphic even function on $\{z \in \mathbb{C}: |\Im(z)| < \frac{1}{2}+\delta \}$ for some $\delta > 0$ such that $h(t) \ll (1+|t|)^{-2-\delta}$, and let
\begin{equation}
g(t) \coloneqq \frac{1}{2\pi} \int_{-\infty}^{\infty} h(u) e^{-itu} \, du.
\end{equation}
The Selberg trace formula in our case reads as follows.
\begin{theorem}[{Cf.~\cites[Theorem~4.1.1]{Fischer1987}[Theorem~6.3]{Hejhal1983}}]\label{thm:Selberg-trace-formula}
Keep the notation as~above. Then
\begin{align}
\sum_{f \in \mathcal{L}_{k}(\Gamma, \nu)} h(t_{f}) &= \frac{\mathrm{vol}(\Gamma \backslash \mathbb{H}_{2}) \dim \nu}{4\pi} \int_{-\infty}^{\infty} t h(t) \frac{\sinh(2\pi t)}{\cosh(2\pi t)+\cos(2\pi k)} \, dt\\
& + \frac{\mathrm{vol}(\Gamma \backslash \mathbb{H}_{2}) \dim \nu}{2\pi} \sum_{\ell = 0}^{\lfloor |k|-\frac{1}{2} \rfloor} \Big(|k|-\ell-\frac{1}{2} \Big) h \Big(i \Big(|k|-\ell-\frac{1}{2} \Big) \Big),\\
& + \sum_{\mathrm{tr}(\gamma) > 2} \frac{\mathrm{tr}(\nu(\gamma)) g(\log \mathrm{N}(\gamma))}{2 \sinh(\frac{\log \mathrm{N}(\gamma)}{2})} \Lambda_{\Gamma}(\gamma)\\
& + \sum_{\mathrm{tr}(\gamma) < 2} \sum_{0 < \theta_{\gamma} < \pi} \frac{\mathrm{tr}(\nu(\gamma)) ie^{i(2k-1)\theta_{\gamma}}}{2|\gamma| \sin \theta_{\gamma}} \int_{-\infty}^{\infty} g(u) e^{(k-\frac{1}{2})u} \frac{e^{u}-e^{2i\theta_{\gamma}}}{\cosh u-\cos 2\theta_{\gamma}} \, du,\\
& - g(0) \Big(h_{0} \dim \nu \log 2+\sum_{\ell = 1}^{D} \sum_{\kappa_{\mathfrak{a}}^{(\ell)} \ne 0} \log|1-e(\kappa_{\mathfrak{a}})| \Big)\\
& + \frac{1}{2} \sum_{\ell = 1}^{D} \sum_{\kappa_{\mathfrak{a}}^{(\ell)} \ne 0} \Big(\frac{1}{2}-\kappa_{\mathfrak{a}}^{(\ell)} \Big) \int_{0}^{\infty} g(u) \frac{\sinh(ku)}{\sinh(\frac{u}{2})} \, du\label{eq:Cauchy}\\
& + \frac{h(0)}{4} \mathrm{tr} \Big(I_{h}-\Phi \Big(\frac{1}{2}, \nu \Big) \Big)+\frac{h}{2} \int_{0}^{\infty} g(u) \frac{1-e^{-ku}}{\sinh(\frac{u}{2})} \, du\\
& + \frac{h}{2\pi} \int_{-\infty}^{\infty} h(t) \frac{\Gamma^{\prime}(1+it)}{\Gamma(1+it)} \, dt+\frac{1}{4\pi} \int_{-\infty}^{\infty} h(t) \frac{\varphi^{\prime}}{\varphi} \Big(\frac{1}{2}+it, \nu \Big) \, dt,\label{eq:Selberg-trace-formula}
\end{align}
where $|\gamma|$ denotes the order of the elliptic element $\gamma$, $\theta_{\gamma}$ denotes the angle in $(0, 2\pi)$ for which $\big(\begin{smallmatrix} \cos \theta_{\gamma} & -\sin \theta_{\gamma} \\ \sin \theta_{\gamma} & \cos \theta_{\gamma} \end{smallmatrix} \big)$ is a $\mathrm{SL}_{2}(\mathbb{R})$-conjugate of $\gamma$, and $h_{0}$ denotes the total number of cusps of $\Gamma$.
\end{theorem}

\subsection{Weyl's law over unit windows}\label{subsect:Weyl}
We are now prepared to formulate Weyl's law,~which provides a unified control over the distribution of the discrete and continuous spectra in an expanding window, despite the availability of a more refined approximation in the literature.
\begin{lemma}[{Cf.~\cites[Theorem~2.28]{Hejhal1983}[Theorem~7.3]{Venkov1990}}]\label{lem:Weyl-general}
Let $\Gamma \subset \mathrm{PSL}_{2}(\mathbb{R})$, and let $\nu: \Gamma \to \mathrm{GL}_{D}(\mathbb{C})$ be a multiplier system of weight $k$ and dimension $D$. Then
\begin{multline}
\#\{f \in \mathcal{L}_{k}(\Gamma, \nu): 0 < t_{f} \leq T \}-\frac{1}{4\pi} \int_{-T}^{T} \frac{\varphi^{\prime}}{\varphi} \Big(\frac{1}{2}+it, \nu \Big) \, dt\\
 = \frac{\mathrm{vol}(\Gamma \backslash \mathbb{H}_{2}) \dim \nu}{4\pi} T^{2}-\frac{h \dim \nu}{\pi} T \log T+O(T).
\end{multline}
\end{lemma}

Moreover, the combined application of \cref{lem:Weyl-general} and the Maa{\ss}--Selberg relation~\cites[Proposition~3.5]{Humphries2017} yields Weyl's law over unit windows.
\begin{lemma}[{Cf.~\cites[Lemma~3.9]{CherubiniGuerreiro2018}}]\label{lem:Weyl-general-unit}
Keep the notation as in \cref{lem:Weyl-general}. Then
\begin{equation}
\#\{f \in \mathcal{L}_{k}(\Gamma, \nu): T < t_{f} \leq T+1 \}-\int_{T < |t| \leq T+1} \Big|\frac{\varphi^{\prime}}{\varphi} \Big(\frac{1}{2}+it, \nu \Big) \Big| \, dt \ll_{\Gamma, \nu} T.
\end{equation}
\end{lemma}

\subsection{Selberg and Ruelle zeta functions}
The Selberg zeta function of a cofinite Fuchsian group $\Gamma \subset \mathrm{SL}_{2}(\mathbb{R})$ associated to a multiplier system $\nu$ of weight $k$ and dimension $D$ is defined by an absolutely convergent Euler product~\cites[Equation~(5.1)]{Hejhal1983}
\begin{equation}\label{eq:Selberg-zeta}
\mathcal{Z}_{\Gamma}(s, \nu) \coloneqq \prod_{\mathrm{tr}(\gamma_{0}) > 2} \prod_{\ell = 0}^{\infty} \det(I_{\dim \nu}-\nu(\gamma_{0}) \mathrm{N}(\gamma_{0})^{-s-\ell}), \qquad \Re(s) > 1-\frac{|k|}{2},
\end{equation}
where the outer product is taken over the set of all primitive hyperbolic conjugacy classes~in $\Gamma$. Note that Hejhal states $\Re(s) > 1$ as a conservative region valid for all multiplier systems, but a more refined analysis using Selberg's spectral theory establishes convergence in the sharper region $\Re(s) > 1-\frac{|k|}{2}$, in conjunction with the main term in the prime geodesic~theorem. The Selberg zeta function admits a meromorphic continuation to $\mathbb{C}$ and has order $2$. Furthermore, without loss of generality, one may assume that $\nu$ is irreducible, since the decomposition of a unitary multiplier system corresponds to that of the Selberg zeta function~\cites[Page~441]{VenkovZograf1982}.

The singularities of the Selberg zeta function $\mathcal{Z}_{\Gamma}(s, \nu)$ are summarised in the comprehensive itemisations in~\cites[Pages~48--49]{Venkov1990}[Page~452]{KanekoKoyama2022}, which include
\begin{itemize}
\item nontrivial zeros at $s = s_{f} = \frac{1}{2}+it_{f}$ and $s = 1-s_{f} = \frac{1}{2}-it_{f}$, where $f \in \mathcal{L}_{k}(\Gamma, \nu)$;
\item nontrivial zeros at $s = 1-\rho$, where $\rho$ denote the zeros of $\varphi(s, \nu)$; cf.~\eqref{eq:s-1-s}.
\end{itemize}

The Ruelle zeta function associated to a multiplier system $\nu$ of weight $k$ and dimension~$D$ on $\Gamma$ is defined analogously by an absolutely convergent Euler product~\cites[Page~232]{Ruelle1976}
\begin{equation}\label{eq:Ruelle-zeta}
\mathcal{R}_{\Gamma}(s, \nu) \coloneqq \prod_{\mathrm{tr}(\gamma_{0}) > 2} \det(I_{\dim \nu}-\nu(\gamma_{0}) \mathrm{N}(\gamma_{0})^{-s}), \qquad \Re(s) > 1-\frac{|k|}{2}.
\end{equation}
It follows from the absolute convergence of~\eqref{eq:Selberg-zeta} and~\eqref{eq:Ruelle-zeta} that
\begin{equation}
\mathcal{R}_{\Gamma}(s, \nu) = \frac{\mathcal{Z}_{\Gamma}(s, \nu)}{\mathcal{Z}_{\Gamma}(s+1, \nu)}, \qquad \Re(s) > 1-\frac{|k|}{2}.
\end{equation}

We now formulate the Hadamard factorisation in the following form.
\begin{lemma}[{Cf.~\cites[Theorem~2.16]{Hejhal1983}}]\label{lem:Hejhal-formal}
The singularities of the logarithmic derivative of the Ruelle zeta function in the strip $0 \leq \Re(s) \leq 1$ are given by the purely formal expression
\begin{equation}
\sum_{\pm} \sum_{f \in \mathcal{L}_{k}(\Gamma, \nu)} \frac{1}{s-\frac{1}{2} \pm it_{f}}+\sum_{\rho} \frac{1}{s+\rho-1}+\frac{\mathrm{tr}(I_{h}-\Phi(\frac{1}{2}, \nu))}{1-2s}.
\end{equation}
\end{lemma}

As an immediate consequence, we deduce an asymptotic version of \cref{lem:Hejhal-formal}.
\begin{lemma}[{Cf.~\cites[Theorem~2.24~(iv)]{Hejhal1983}}]\label{lem:Hejhal-asymptotic}
Let $s = \sigma+i\tau$. Then we have uniformly in every vertical strip of bounded width that
\begin{equation}
\frac{\mathcal{R}_{\Gamma}^{\prime}}{\mathcal{R}_{\Gamma}}(s, \nu) = \sum_{\pm} \sum_{|\tau \pm t_{f}| \leq 1} \frac{1}{s-\frac{1}{2} \pm it_{f}}+\sum_{|\tau-\Im(\rho)| \leq 1} \frac{1}{s+\rho-1}+O(1+|\tau|).
\end{equation}
\end{lemma}

The following direct corollary serves as the counterpart of~\cites[Equations~(22) and~(24)]{Iwaniec1984}.
\begin{lemma}[{Cf.~\cites[Corollary~3.1]{Matthes1994}}]\label{lem:Ruelle-convexity}
Let $s = \sigma+i\tau$. Then we have for any $\varepsilon > 0$ that
\begin{equation}
\frac{\mathcal{R}_{\Gamma}^{\prime}}{\mathcal{R}_{\Gamma}} \Big(1-\frac{|k|}{2}+\varepsilon+i\tau, \nu \Big) \ll \frac{1}{\varepsilon}, \qquad \frac{\mathcal{R}_{\Gamma}^{\prime}}{\mathcal{R}_{\Gamma}}(-\varepsilon+i\tau, \nu) \ll 1+|\tau|.
\end{equation}
\end{lemma}

\section{Proof of \texorpdfstring{\cref{thm:main}}{}}\label{sect:proof-main}
In contrast to the rich knowledge when $\nu$ is trivial, the Brun--Titchmarsh-type theorem over short intervals in an asymptotic form, or even an upper bound, is by no means automatic for nontrivial $\nu$ due to technical complications caused by the necessity to write down $\nu$ explicitly. Nonetheless, the following crude yet general result suffices for our subsequent purposes.
\begin{proposition}\label{prop:Brun-Titchmarsh}
Let $\Gamma \subset \mathrm{PSL}_{2}(\mathbb{R})$, and let $\nu: \Gamma \to \mathrm{GL}_{D}(\mathbb{C})$ be a unitary multiplier system of weight $k \in (-1, 1]$ and dimension $D$. Then we have for any $x^{\frac{1+|k|}{2}} \leq y \leq x$ that
\begin{equation}
\Psi_{\Gamma}^{(2)}(x+y, \nu)-\Psi_{\Gamma}^{(2)}(x, \nu) = \mathop{\sum \sum}_{\substack{f \in \mathcal{L}_{k}(\Gamma, \nu) \\ \frac{1}{2} < s_{f} \leq 1-\frac{|k|}{2}}} \sum_{\ell = 1}^{\infty} \binom{s_{f}}{\ell} \frac{x^{s_{f}-\ell} y^{\ell}}{s_{f}}+O_{\Gamma, \nu, \varepsilon}(x^{\delta_{1}^{(2)}(\nu)+\varepsilon}).
\end{equation}
\end{proposition}

\begin{proof}
The claim follows upon invoking \cref{def:delta-k-nu} and subtracting~\eqref{eq:Hejhal} from the one~with $x \mapsto x+y$ via the binomial expansion.
\end{proof}

We are prepared to establish the spectral explicit formula \`{a}~la Iwaniec~\cites[Lemma~1]{Iwaniec1984}.
\begin{proposition}\label{prop:spectral-explicit-formula}
Let $\Gamma \subset \mathrm{PSL}_{2}(\mathbb{R})$, and let $\nu: \Gamma \to \mathrm{GL}_{D}(\mathbb{C})$ be a unitary multiplier system of weight $k \in (-1, 1]$ and dimension $D $. Then we have for any $1 \leq T \leq x^{\frac{1-|k|}{2}}$ that
\begin{equation}\label{eq:spectral-explicit-formula}
\mathcal{E}_{\Gamma}^{(2)}(x, \nu) = \sum_{\pm} \mathop{\sum \sum}_{\substack{f \in \mathcal{L}_{k}(\Gamma, \nu) \\ 0 < t_{f} \leq T}} \frac{x^{\frac{1}{2} \pm it_{f}}}{\frac{1}{2} \pm it_{f}}+\sum_{0 < |\Im(\rho)| \leq T} \frac{x^{1-\rho}}{1-\rho}+O_{\Gamma, \nu, \varepsilon} \Big(\frac{x^{1-\frac{|k|}{2}+\varepsilon}}{T} \Big).
\end{equation}
\end{proposition}

\begin{proof}
It is convenient to follow the techniques in~\cites[Section~5]{Iwaniec1984}[Lemma~2.3]{KanekoKoyama2022}.~By Perron's formula, the left-hand side of~\eqref{eq:spectral-explicit-formula} is equal to
\begin{equation}\label{eq:Perron}
\int_{\sigma-iT}^{\sigma+iT} \frac{\mathcal{R}_{\Gamma}^{\prime}}{\mathcal{R}_{\Gamma}}(s, \nu) \frac{x^{s}} s \frac{ds}{2\pi i}+O \Big(x^{\sigma} \Big|\sum_{\mathrm{tr}(\gamma) > 2} \min \Big\{1, \Big(T \Big|\log \frac{x}{\mathrm{N}(\gamma)} \Big| \Big)^{-1} \Big\} \frac{\mathrm{tr}(\nu(\gamma)) \Lambda_{\Gamma}(\gamma)}{\mathrm{N}(\gamma)^{\sigma}} \Big| \Big),
\end{equation}
where $\sigma > 1-\frac{|k|}{2}$ and $x \ne \mathrm{N}(\gamma)$. By \cref{prop:Brun-Titchmarsh},~\eqref{eq:3/4}, partial summation, and the~mean value theorem, the contribution of the sums over $x < \mathrm{N}(\gamma) \leq 2x$ and $\frac{x}{2} \leq \mathrm{N}(\gamma) < x$ leads to
\begin{equation}
\ll_{\Gamma, \nu} \frac{(x^{1-\frac{|k|}{2}}+x^{\frac{3}{4}})(\log x)^{2}}{T}.
\end{equation}
It remains to evaluate the sums over $\mathrm{N}(\gamma) < \frac{x}{2}$ and $2x < \mathrm{N}(\gamma)$. Since $|\log x-\log \mathrm{N}(\gamma)|^{-1} \ll 1$ in these ranges, each contribution is bounded trivially by (cf.~\cites[Equation~(22)]{Iwaniec1984})
\begin{equation}
\ll \frac{x^{\sigma} \log x}{(\sigma-1+\frac{|k|}{2})T}
\end{equation}
 upon choosing $\sigma = 1-\frac{|k|}{2}+(\log T)^{-1}$, which aligns with the error term in~\eqref{eq:spectral-explicit-formula}.

Following~\cites[Pages~145--147]{Iwaniec1984} \textit{mutatis mutandis}, we now shift the contour of integration in~\eqref{eq:Perron}, detect the polar terms via \cref{lem:Hejhal-formal}, and make use of \cref{lem:Ruelle-convexity} to verify that~the contribution of each segment is negligibly small. Consequently, the above choice of $\sigma$ implies
\begin{multline}
\Psi_{\Gamma}^{(2)}(x, \nu) = \mathop{\sum \sum}_{\substack{f \in \mathcal{L}_{k}(\Gamma, \nu) \\ \frac{1}{2} < s_{f} \leq 1-\frac{|k|}{2}}} \frac{x^{s_{f}}}{s_{f}}+\sum_{\pm} \mathop{\sum \sum}_{\substack{f \in \mathcal{L}_{k}(\Gamma, \nu) \\ 0 < t_{f} \leq T}} \frac{x^{\frac{1}{2} \pm it_{f}}}{\frac{1}{2} \pm it_{f}}+\sum_{0 < |\Im(\rho)| \leq T} \frac{x^{1-\rho}}{1-\rho}\\
 + O_{\Gamma, \nu} \Big(T+\frac{(x^{1-\frac{|k|}{2}}+x^{\frac{3}{4}})(\log x)^{2}}{T}+x^{\frac{1}{2}}(\log x)^{4} \Big).
\end{multline}
By induction, the term $x^{\frac{3}{4}}$ in the innermost parentheses is absorbed into the first term $x^{1-\frac{|k|}{2}}$ as follows. If the second term dominates, namely if $|k| > \frac{1}{2}$, then $\delta_{1}^{(2)}(\nu) \leq \frac{3}{4}$ as in~\eqref{eq:3/4} along with the optimisation $T = x^{\frac{1}{4}} \log x$ and \cref{lem:Weyl-general} implies that $\delta_{1}^{(2)}(\nu) \leq \frac{5}{8}$. If the improved second term still dominates, namely if $|k| > \frac{3}{4}$, then $\delta_{1}^{(2)}(\nu) \leq \frac{5}{8}$ along with the optimisation $T = x^{\frac{1}{16}} \log x$ and \cref{lem:Weyl-general} implies that $\delta_{1}^{(2)}(\nu) \leq \frac{9}{16}$. The iteration of this~process $j$ times implies $\delta_{1}^{(2)}(\nu) \leq \frac{1}{2}+\frac{1}{2^{j+2}}$. Taking $j$ large thus completes the proof of \cref{prop:spectral-explicit-formula}.
\end{proof}

\begin{proof}[Proof of \cref{thm:main}]
It follows from \cref{prop:spectral-explicit-formula},~\cref{lem:Weyl-general}, and partial summation that the right-hand side of~\eqref{eq:spectral-explicit-formula} is bounded by
\begin{equation}
O_{\Gamma, \nu, \varepsilon} \Big(x^{\frac{1}{2}} T+\frac{x^{1-\frac{|k|}{2}+\varepsilon}}{T} \Big).
\end{equation}
Optimising $T = x^{\frac{1-|k|}{4}}$ now justifies the required statement.
\end{proof}

\section{Proof of \texorpdfstring{\cref{thm:CG}}{}}\label{sect:proof-main-2}
We follow the approach in~\cites[Section~2]{CherubiniGuerreiro2018} or~\cites[Section~4]{BalkanovaChatzakosCherubiniFrolenkovLaaksonen2019}. If $\rho \coloneqq \log x$, then a naive choice of the test function in~\eqref{eq:Selberg-trace-formula} boils down to
\begin{equation}\label{eq:g-rho}
g_{\rho}(u) \coloneqq 2 \sinh \Big(\frac{|u|}{2} \Big) \delta_{0 \leq |u| \leq \rho},
\end{equation}
since the hyperbolic term in~\eqref{eq:Selberg-trace-formula} equals $\Psi_{\Gamma}^{(2)}(x, \nu)$. Nonetheless, the function $g_{\rho}(u)$ does not qualify as an admissible test function; hence, it becomes requisite to construct an appropriate approximation thereof. A delicate treatment of~\eqref{eq:Selberg-trace-formula} would then lead to the desideratum.

We now construct two functions $g_{\pm}(u)$ approximating from above and below the function $g_{\rho}(u)$ defined by~\eqref{eq:g-rho}, which are in turn admissible. Let $q(u)$ be an even smooth nonnegative function on $\mathbb{R}$ satisfying $\mathrm{supp}(q) \subseteq [-1, 1]$ and unit mass $\norm{q}_{1} = 1$. For $0 < \delta < \frac{1}{4}$, let
\begin{equation}
q_{\delta}(u) \coloneqq \frac{1}{\delta} q \Big(\frac{x}{\delta} \Big).
\end{equation}
Given $\rho > 1$, we define $g_{\pm}(u)$ to be the convolution of $g_{\rho \pm \delta}(u)$ with $q_{\delta}(u)$, namely
\begin{equation}
g_{\pm}(u) \coloneqq (g_{\rho \pm \delta} \ast q_{\delta})(u) = \int_{-\infty}^{\infty} g_{\rho \pm \delta}(u-v) q_{\delta}(v) \, dv.
\end{equation}
If $\hat{q}_{\delta}$ denotes the Fourier transform of $q_{\delta}$ and $h_{\pm}$ denotes the Fourier transform of $g_{\pm}$, then
\begin{equation}
h_{\pm}(t) = h_{\rho \pm \delta}(t) \hat{q}_{\delta}(t).
\end{equation}
By~\cites[Lemma~3.1]{CherubiniGuerreiro2018}, $h_{\pm}(t)$ obeys the assumptions imposed in \cref{thm:Selberg-trace-formula}, and therefore $g_{\pm}(u)$ is admissible. By construction, we have $\mathrm{supp}(g_{\pm}) \subseteq [0, \rho+\delta \pm \delta]$. Furthermore,~\cites[Lemma~3.2]{CherubiniGuerreiro2018} demonstrates that for any $u \geq 0$,
\begin{equation}
g_{-}(u)+O(\delta u^{\frac{x}{2}} \delta_{0 \leq u \leq \rho}) \leq g_{\rho}(u) \leq g_{+}(u)+O(\delta e^{\frac{u}{2}} \delta_{0 \leq u \leq \rho}).
\end{equation}
If one abbreviates
\begin{equation}
\Psi_{\Gamma, \pm}^{(2)}(x, \nu) \coloneqq \sum_{\mathrm{tr}(\gamma) > 2} \frac{\mathrm{tr}(\nu(\gamma)) g_{\pm}(\log \mathrm{N}(\gamma))}{2 \sinh(\frac{\log \mathrm{N}(\gamma)}{2})} \Lambda_{\Gamma}(\gamma),
\end{equation}
then the trivial bound $\Psi_{\Gamma}^{(2)}(x, \nu) \ll x^{1-\frac{|k|}{2}}$ implies
\begin{equation}
\Re(\Psi_{\Gamma, -}^{(2)}(x, \nu))+O(\delta x^{1-\frac{|k|}{2}}) \leq \Re(\Psi_{\Gamma}^{(2)}(x, \nu)) \leq \Re(\Psi_{\Gamma, +}^{(2)}(x, \nu))+O(\delta x^{1-\frac{|k|}{2}}).
\end{equation}
It follows from these observations that
\begin{equation}\label{eq:reduction-to-+-}
\frac{1}{Y} \int_{X}^{X+Y} |\mathcal{E}_{\Gamma}^{(2)}(x, \nu)|^{2} \, dx \ll \frac{1}{Y} \int_{X}^{X+Y} |\Psi_{\Gamma, \pm}^{(2)}(x, \nu)-\mathcal{M}_{\Gamma}^{(2)}(x, \nu)|^{2} \, dx+O(\delta^{2} X^{2-|k|}),
\end{equation}
where
\begin{equation}\label{eq:M}
\mathcal{M}_{\Gamma}^{(2)}(x, \nu) \coloneqq \mathop{\sum \sum}_{\substack{f \in \mathcal{L}_{k}(\Gamma, \nu) \\ \frac{1}{2} < s_{f} \leq 1-\frac{|k|}{2}}} \frac{x^{s_{f}}}{s_{f}}.
\end{equation}
At this point, Cherubini and Guerreiro~\cites[Equation~(2-12)]{CherubiniGuerreiro2018} pass to the logarithmic~scale, but the polynomial scale as in~\cites[Section~4]{BalkanovaChatzakosCherubiniFrolenkovLaaksonen2019} is more convenient and straightforward in our setting. The substitution of a test function as in~\cites[Equation~(2-11)]{CherubiniGuerreiro2018} is unnecessary.

In the contribution of the discrete spectrum in~\eqref{eq:Selberg-trace-formula}, the treatment of the small eigenvalues $\lambda_{f} \in [\frac{|k|}{2}(1-\frac{|k|}{2}), \frac{1}{4})$ requires particular care. For notational simplicity, we write $\mathcal{M}_{\Gamma, \pm}^{(2)}(x, \nu, \rho)$ for the sum over such Laplace eigenvalues, namely
\begin{equation}
\mathcal{M}_{\Gamma, \pm}^{(2)}(x, \nu, \rho) \coloneqq \mathop{\sum \sum}_{\substack{f \in \mathcal{L}_{k}(\Gamma, \nu) \\ t_{f} \in i(0, \frac{1-|k|}{2}]}} h_{\pm}(t_{f}).
\end{equation}
The following lemma generalises~\cites[Lemma~3.3]{CherubiniGuerreiro2018}.
\begin{lemma}\label{lem:h+-}
Keep the notation and assumptions as above. Let $t_{f} \in i(0, \frac{1-|k|}{2}]$ be the spectral parameter corresponding to the Laplace eigenvalue $\lambda_{f} \in [\frac{|k|}{2}(1-\frac{|k|}{2}), \frac{1}{4})$. Then there exists a constant $0 < \varepsilon \leq \frac{1}{4}$ depending at most on $(\Gamma, \nu)$ such that
\begin{equation}\label{eq:h+-}
h_{\pm}(t_{f}) = \frac{x^{\frac{1}{2}+|t_{f}|}}{\frac{1}{2}+|t_{f}|}+O(\delta x^{1-\frac{|k|}{2}}+x^{\frac{1}{2}-\varepsilon}).
\end{equation}
\end{lemma}

\begin{proof}
A key distinction from the proof of~\cites[Lemma~3.3]{CherubiniGuerreiro2018} is that the spectral parameter corresponding to the bottom eigenvalue $\lambda_{0} = \frac{|k|}{2}(1-\frac{|k|}{2})$ is $\frac{(1-|k|)i}{2}$ in place of $\frac{i}{2}$, affecting~the size of the first term in the error term in~\eqref{eq:h+-}. The proof of \cref{lem:h+-} is thus complete.
\end{proof}

Consequently, \cref{lem:h+-} leads to the approximation
\begin{equation}\label{eq:M+-asymptotic}
\mathcal{M}_{\Gamma, \pm}^{(2)}(x, \nu) =\mathcal{M}_{\Gamma}^{(2)}(x, \nu)+O(\delta x^{1-\frac{|k|}{2}}+x^{\frac{1}{2}}).
\end{equation}
For the other contributions, one may execute minor adjustments of~\cites[Lemmata~3.5--3.7]{CherubiniGuerreiro2018} to derive the same bounds. Furthermore, concerning the contributions from the discrete and continuous spectra, we note that~\cites[Propositions~3.10~and~3.11]{CherubiniGuerreiro2018} are attributed to~\cites[Lemma~3.8]{CherubiniGuerreiro2018}. As a substitute, the following result analogous to~\cites[Lemma~4.1]{BalkanovaChatzakosCherubiniFrolenkovLaaksonen2019} plays~a crucial role, with its proof strategy being identical to that of~\cites[Lemma~3.8]{CherubiniGuerreiro2018}.
\begin{lemma}\label{lem:Balkanova-et-al-analogue}
Keep the notation and assumptions as above. If $(t_{1}, t_{2}) \in \mathbb{R}^{2}$, then
\begin{equation}
\frac{1}{Y} \int_{X}^{X+Y} h_{\rho \pm \delta}(t_{1}) h_{\rho \pm \delta}(t_{2}) \, dx \ll \frac{X^{2}}{Y} v(t_{1}) v(t_{2}) v(|t_{1}|-|t_{2}|)+\frac{X^{\frac{3}{2}}}{Y} v(t_{1}^{2}) v(t_{2}^{2}),
\end{equation}
where $v(t) \coloneqq (1+|t|)^{-1}$, and the implicit constant is independent of $\delta$.
\end{lemma}

By~\eqref{eq:M+-asymptotic} and \cref{lem:Balkanova-et-al-analogue}, the counterpart of the first display on~\cites[Page~580]{CherubiniGuerreiro2018} reads
\begin{equation}
\frac{1}{Y} \int_{X}^{X+Y} |\Psi_{\Gamma, \pm}^{(2)}(x, \nu)-\mathcal{M}_{\Gamma}^{(2)}(x, \nu)|^{2} \, dx \ll \frac{X^{2}}{\delta Y}+\frac{X^{\frac{3}{2}}}{Y} \log^{2} \frac{1}{\delta}+\delta^{2} X^{2-|k|}.
\end{equation}
Optimising $\delta = Y^{-\frac{1}{3}} X^{\frac{|k|}{3}}$ now yields\footnote{The argument $s$ should read $\rho$ in the second display on~\cites[Page~580]{CherubiniGuerreiro2018}.}
\begin{equation}
\frac{1}{Y} \int_{X}^{X+Y} |\Psi_{\Gamma, \pm}^{(2)}(x, \nu)-\mathcal{M}_{\Gamma}^{(2)}(x, \nu)|^{2} \, dx \ll X^{\frac{4-|k|}{3}} \Big(\frac{X}{Y} \Big)^{\frac{2}{3}},
\end{equation}
as required. Now, \cref{thm:CG} is immediate from~\eqref{eq:reduction-to-+-} and~\eqref{eq:def-2}.

\section{Proof of \texorpdfstring{\cref{thm:mean-to-max}}{}}
It is now convenient to invoke \cref{def:delta-k-nu} and the notation introduced in~\eqref{eq:Hejhal} and~\eqref{eq:M}. One may suppose that there exists a constant $c > 0$ depending at most on $(\Gamma, \nu)$ such that
\begin{equation}
|\Re(\mathcal{E}_{\Gamma}^{(2)}(x, \nu))| \geq cx^{\delta_{1}^{(2)}(\nu)}
\end{equation}
holds for infinitely many $x \geq 2$ as $x \to \infty$. Without loss of generality, one may suppose that $c \geq 1$ and $\Re(\mathcal{E}_{\Gamma}^{(2)}(x, \nu))$ is nonnegative for those values of $x$. Hence, if
\begin{equation}\label{eq:if}
X \leq x \leq X+X^{\delta_{1}^{(2)}(\nu)+\frac{|k|}{2}},
\end{equation}
then the binomial expansion implies
\begin{align}
\Re(\mathcal{E}_{\Gamma}^{(2)}(x, \nu)) &= \Re(\Psi_{\Gamma}^{(2)}(x, \nu))-\mathcal{M}_{\Gamma}^{(2)}(x, \nu)\\
&\geq \mathcal{M}_{\Gamma}^{(2)}(X, \nu)+\Re(\mathcal{E}_{\Gamma}^{(2)}(X, \nu))-\mathcal{M}_{\Gamma}^{(2)}(x, \nu)\\
&\geq \mathcal{M}_{\Gamma}^{(2)}(X, \nu)+X^{\delta_{1}^{(2)}(\nu)}-\mathcal{M}_{\Gamma}^{(2)}(x, \nu) \geq 0.\label{eq:inequalities}
\end{align}
Integrating the lower bound~\eqref{eq:inequalities} with respect to $x$ yields
\begin{align}
&\frac{1}{X^{\delta_{1}^{(2)}(\nu)+\frac{|k|}{2}}} \int_{X}^{X+X^{\delta_{1}^{(2)}(\nu)+\frac{|k|}{2}}} |\mathcal{E}_{\Gamma}^{(2)}(x, \nu)|^{2} \, dx\\
& = \frac{1}{X^{\delta_{1}^{(2)}(\nu)+\frac{|k|}{2}}} \int_{X}^{X+X^{\delta_{1}^{(2)}(\nu)+\frac{|k|}{2}}} (\Re(\mathcal{E}_{\Gamma}^{(2)}(x, \nu))^{2}+\Im(\mathcal{E}_{\Gamma}^{(2)}(x, \nu))^{2}) \, dx\\
& \geq \frac{1}{X^{\delta_{1}^{(2)}(\nu)+\frac{|k|}{2}}} \int_{X}^{X+X^{\delta_{1}^{(2)}(\nu)+\frac{|k|}{2}}} (\mathcal{M}_{\Gamma}^{(2)}(X, \nu)+X^{\delta_{1}^{(2)}(\nu)}-\mathcal{M}_{\Gamma}^{(2)}(x, \nu))^{2} \, dx \gg X^{2\delta_{1}^{(2)}(\nu)}.\label{eq:integrate-lower-bound}
\end{align}
The combination of~\eqref{eq:integrate-lower-bound} and~\eqref{eq:def-2} with $Y = X^{\delta_{1}^{(2)}(\nu)+\frac{|k|}{2}}$ implies
\begin{equation}
X^{2\delta_{1}^{(2)}(\nu)} \ll_{\Gamma, \nu, \varepsilon} X^{2\delta_{2}^{(2)}(\nu)+(2-|k|)\eta_{2}^{(2)}(\nu)-2\delta_{1}^{(2)}(\nu) \eta_{2}^{(2)}(\nu)+\varepsilon},
\end{equation}
yielding \cref{thm:mean-to-max} by solving for $\delta_{1}^{(2)}(\nu)$.

\section{Proof of \texorpdfstring{\cref{thm:pointwise}}{}}\label{sect:proof}

\subsection{Specialisation to the theta multiplier}\label{subsect:specialisation}
Let $\Gamma \coloneqq \mathrm{SL}_{2}(\mathbb{Z})$, and let~\cites[Page~226]{Matthes1993-2}
\begin{align}
\Gamma_{0}(q) &\coloneqq \{\gamma \in \Gamma: \gamma \equiv \begin{psmallmatrix} \ast & \ast \\ 0 & \ast \end{psmallmatrix} \tpmod{q} \},\\
\Gamma_{\vartheta}(q) &\coloneqq \{\gamma \in \Gamma: \gamma \equiv \begin{psmallmatrix} 1 & 0 \\ 0 & 1 \end{psmallmatrix} \vee \begin{psmallmatrix} 0 & -1 \\ 1 & 0 \end{psmallmatrix} \tpmod{q} \},\\
\Gamma^{0}(q) &\coloneqq \{\gamma \in \Gamma: \gamma \equiv \begin{psmallmatrix} \ast & 0 \\ \ast & \ast \end{psmallmatrix} \tpmod{q} \},
\end{align}
then the theta functions of significance to the proof of \cref{thm:pointwise} are given by
\begin{equation}
\vartheta_{2}(z) \coloneqq \sum_{n \in \mathbb{Z}} e^{\pi i(n+\frac{1}{2})^{2} z}, \qquad \vartheta_{3}(z) \coloneqq \sum_{n \in \mathbb{Z}} e^{\pi in^{2} z}, \qquad \vartheta_{4}(z) \coloneqq \vartheta_{3}(z+1),
\end{equation}
which are holomorphic modular forms of weight $\frac{1}{2}$ and multiplier systems $\nu_{2}$, $\nu_{3}$, $\nu_{4}$ on $\Gamma_{0}(2)$, $\Gamma_{\vartheta}(2)$, $\Gamma^{0}(2)$, respectively~\cite[Theorem~7.1.3]{Rankin1977}. Following~\cites[Equation~(2.9)]{Selberg1965}, if one forms a vector $\Theta(z) \coloneqq (\vartheta_{2}(z), \vartheta_{3}(z), \vartheta_{4}(z))^{\mathrm{T}}$, then it is a holomorphic modular form of weight $\frac{1}{2}$ on $\Gamma$, and the corresponding multiplier system $\nu_{\Theta}: \Gamma \mapsto \mathrm{Aut}(\mathbb{C}^{3})$ is called the $3$-fold theta multiplier. By routine considerations, it evaluates to
\begin{equation}\label{eq:nu-theta}
\nu_{\Theta}(\gamma) = \rho(\gamma) \Bigg(\begin{matrix} \nu_{2}^{\ast}(\gamma) & 0 & 0 \\ 0 & \nu_{3}^{\ast}(\gamma) & 0 \\ 0 & 0 & \nu_{4}^{\ast}(\gamma) \end{matrix} \Bigg),
\end{equation}
where $\nu_{i}^{\ast}$ denote certain extensions of $\nu_{i}$ to $\Gamma$, and $\rho$ denotes a unitary representation (or a permutation matrix) of $\mathrm{SL}_{2}(\mathbb{F}_{2})$ in $\mathbb{C}^{3}$ determined by its values on the generators
\begin{equation}
\rho \Big(\begin{matrix} 1 & 1 \\ 0 & 1 \end{matrix} \Big) = \Bigg(\begin{matrix} 1 & 0 & 0 \\ 0 & 0 & 1 \\ 0 & 1 & 0 \end{matrix} \Bigg), \qquad 
\rho \Big(\begin{matrix} 0 & -1 \\ 1 & 0 \end{matrix} \Big) = \Bigg(\begin{matrix} 0 & 0 & 1 \\ 0 & 1 & 0 \\ 1 & 0 & 0 \end{matrix} \Bigg).
\end{equation}
It is well known that $\nu_{\Theta}$ is irreducible. Automorphy~\eqref{eq:automorphy} shows that $\Theta(z)$ satisfies
\begin{equation}
\Theta(\gamma z) = (cz+d)^{\frac{1}{2}} \nu_{\Theta}(\gamma) \Theta(z).
\end{equation}
Furthermore, each multiplier system $\nu_{i}$ admits an explicit realisation~\cite[Satz~E.~2]{Petersson1982}
\begin{equation}
\begin{alignedat}{2}
\nu_{2}(\gamma) &= \Big(\frac{c}{d} \Big) e \Big(\frac{d-1+bd}{8} \Big), && \qquad \gamma \in \Gamma_{0}(2),\\
\nu_{3}(\gamma) &= \Big(\frac{c}{d} \Big) e \Big(\frac{d-1}{8} \Big), && \qquad \gamma \equiv \Big(\begin{matrix} 1 & 0 \\ 0 & 1 \end{matrix} \Big) \tpmod{2},\\
\nu_{3}(\gamma) &= \Big(\frac{d}{c} \Big) e \Big(-\frac{c}{8} \Big), && \qquad \gamma \equiv \Big(\begin{matrix} 0 & -1 \\ 1 & 0 \end{matrix} \Big) \tpmod{2},\\
\nu_{4}(\gamma) &= \nu_{3} \Big(\Big(\begin{matrix} 1 & 1 \\ 0 & 1 \end{matrix} \Big)^{-1} \gamma \Big(\begin{matrix} 1 & 1 \\ 0 & 1 \end{matrix} \Big) \Big), && \qquad \gamma \in \Gamma^{0}(2),
\end{alignedat}
\end{equation}
where $\gamma = \begin{psmallmatrix} a & b \\ c & d \end{psmallmatrix} \in \Gamma$, and $(\frac{\cdot}{\cdot})$ stands for the extended Jacobi symbol~\cites[Page~442--443]{Shimura1973}.

\subsection{Shimura correspondence}
The following identification plays a key role in the~proof of \cref{thm:pointwise}. The circumstance differs from the case of the  standard multiplier system on $\Gamma_{0}(4)$ where restriction to the Kohnen plus space proves indispensable~\cites[Theorem~1.2]{BaruchMao2010}.
\begin{theorem}\label{thm:Shimura-vector-valued}
Keep the notation as above. There exists a bijective correspondence between $f \in \mathcal{L}_{0}(\Gamma, \mathbbm{1})$ and $\tilde{f} \in \mathcal{L}_{1/2}(\Gamma, \nu_{\Theta})$. In particular, we have that $s_{f}-\frac{1}{2} = 2(s_{\tilde{f}}-\frac{1}{2})$.
\end{theorem}

\begin{proof}
The assertion holds because the Eichler--Zagier correspondence~\cites{EichlerZagier1985} or the Shimura correspondence~\cites{Shimura1973} is given in such a way that the lifting process involving the $3$-fold~theta multiplier preserves the Fourier--Whittaker coefficients with the required normalisation of the spectral parameters by the factor of $2$ across the entire space of vector-valued modular forms under consideration.
\end{proof}

\subsection{Arithmetic explicit formula}
For completeness, we present \cref{prop:spectral-explicit-formula} tailored to the $3$-fold theta multiplier $\nu_{\Theta}$, thereby reducing the analysis of the prime geodesic theorem to estimating nontrivially the spectral exponential sum of half-integral weight.
\begin{proposition}\label{prop:explicit-formula}
Keep the notation as above. Then we have for any $1 \leq T \leq x^{\frac{1}{4}}$ that
\begin{equation}
\mathcal{E}_{\Gamma}^{(2)}(x, \nu_{\Theta}) = \sum_{\pm} \mathop{\sum \sum}_{\substack{\tilde{f} \in \mathcal{L}_{1/2}(\Gamma, \nu_{\Theta}) \\ 0 < t_{\tilde{f}} \leq T}} \frac{x^{\frac{1}{2} \pm it_{\tilde{f}}}}{\frac{1}{2} \pm it_{\tilde{f}}}+O_{\varepsilon} \Big(\frac{x^{\frac{3}{4}+\varepsilon}}{T} \Big).
\end{equation}
\end{proposition}

\begin{proof}
The claim follows from \cref{prop:spectral-explicit-formula},~\cites[Theorem~2.15]{Bruggeman1986}, and the penultimate display on~\cites[Page~174]{Matthes1994}.
\end{proof}

\cref{prop:explicit-formula} differs from~\cites[Proposition~3.3]{Matthes1994}, since it handles a modified counting function \cites[Equation~(11)]{Matthes1994} to maintain consistency with the notation of Hejhal~\cites[Definition~3.2]{Hejhal1983}. Nonetheless, partial summation guarantees a seamless transition between the two, with no alteration in the main term required. Moreover, we highlight that the limitation of~\cites[Proposition~3.3]{Matthes1994} constitutes the suboptimal lower bound $\delta_{1}^{(2)}(\nu_{\Theta}) \geq \frac{9}{16}$.

\begin{proposition}\label{prop:Luo-Sarnak-analogue}
Let $T, X \geq 2$. Then
\begin{equation}
\mathop{\sum \sum}_{\substack{\tilde{f} \in \mathcal{L}_{1/2}(\Gamma, \nu_{\Theta}) \\ 0 < t_{\tilde{f}} \leq T}} X^{it_{\tilde{f}}} \ll_{\varepsilon} T^{\frac{5}{4}} X^{\frac{1}{16}}(\log X)^{2}.
\end{equation}
\end{proposition}

\begin{proof}
\cref{thm:Shimura-vector-valued} ensures the change of variables $X \mapsto X^{\frac{1}{2}}$ in~\cites[Equation~(58)]{LuoSarnak1995}.
\end{proof}

\begin{proof}[Proof of \cref{thm:pointwise}]
It follows from \cref{prop:explicit-formula,prop:Luo-Sarnak-analogue} and partial summation that
\begin{equation}
\mathcal{E}_{\Gamma}^{(2)}(x, \nu_{\Theta}) \ll_{\varepsilon} T^{\frac{1}{4}} x^{\frac{9}{16}+\varepsilon}+\frac{x^{\frac{3}{4}+\varepsilon}}{T}.
\end{equation}
Optimising $T = x^{\frac{3}{20}}$ now justifies the desideratum.
\end{proof}

\section{Proof of \texorpdfstring{\cref{thm:main-metaplectic}}{}}
In preparation for the proof of \cref{thm:main-metaplectic} in \cref{subsect:proof-of-theorem-1-7}, which is of an elementary nature, \cref{subsect:group-action,subsect:cusp,subsect:Shimura-2} provide an overview of fundamental aspects of the metaplectic theory, with particular emphasis on the Shimura correspondence. Note that metaplectic forms themselves are largely extraneous to our discussion, as the metaplectic explicit formula (\cref{prop:explicit-formula-metaplectic}) does not involve Fourier coefficients but rather depends only on the spectral parameters. The interested reader is directed to~\cites{Kubota1968}{Kubota1969} for the theoretical background.

\subsection{Group action and Laplacian}\label{subsect:group-action}
Let $\mathbb{H}_{3} \coloneqq \mathbb{C} \times \mathbb{R}_{+}^{\times}$ denote the quaternionic hyperbolic space. One embeds $\mathbb{C}$ and $\mathbb{H}_{3}$ in the Hamilton quaternions by identifying $i = \sqrt{-1}$ with $\hat{i}$ and $w = (z(w), v(w)) = (z, v) = (x+iy, v) \in \mathbb{H}_{3}$ with $x+y \hat{i}+v \hat{j}$, where $\{1, \hat{i}, \hat{j}, \hat{k} \}$ denotes the~set of the unit quaternions. A typical point $w = (z, v) \in \mathbb{H}_{3}$ is represented by $w = \begin{psmallmatrix} z & -v \\ v & \overline{z} \end{psmallmatrix}$, and a complex number $u \in \mathbb{C}$ is represented by $\tilde{u} = \begin{psmallmatrix} u & 0 \\ 0 & \overline{u} \end{psmallmatrix}$. The discontinuous action of $\mathrm{SL}_{2}(\mathbb{C})$~on $\mathbb{H}_{3}$ in quaternion arithmetic is then given by linear fractional transformations
\begin{equation}\label{eq:linear-fractional}
\gamma w = (\tilde{a} w+\tilde{b})(\tilde{c} w+\tilde{d})^{-1}, \qquad \gamma = \Big(\begin{matrix} a & b \\ c & d \end{matrix} \Big) \in \mathrm{SL}_{2}(\mathbb{C}), \qquad w \in \mathbb{H}_{3}.
\end{equation}
In coordinates, this map sends
\begin{equation}
\gamma w = \Big(\frac{(az+b) \overline{(cz+d)}+a \overline{c} v^{2}}{|cz+d|^{2}+|c|^{2} v^{2}}, \frac{v}{|cz+d|^{2}+|c|^{2} v^{2}} \Big), \qquad w = (z, v).
\end{equation}
A natural left invariant metric on $\mathbb{H}_{3}$ is given by $ds^{2} \coloneqq v^{-2}(dx^{2}+dy^{2}+dv^{2})$, along with~the corresponding volume element $d\mu(w) \coloneqq v^{-3} \, dx \, dy \, dv$. Furthermore, the hyperbolic Laplacian on $\mathbb{H}_{3}$ is given by
\begin{equation}\label{eq:Laplacian-3-dim}
\Delta \coloneqq v^{2} \Big(\frac{\partial^{2}}{\partial x^{2}}+\frac{\partial^{2}}{\partial y^{2}}+\frac{\partial^{2}}{\partial v^{2}} \Big)-v \frac{\partial}{\partial v}.
\end{equation}

\subsection{Metaplectic cusp forms}\label{subsect:cusp}
Let $\Gamma^{\prime} \subseteq \Gamma \coloneqq \mathrm{PSL}_{2}(\omega)$ denote a subgroup with finite~index, and let $\chi: \Gamma^{\prime} \rightarrow \mathbb{C}^{\times}$ be a unitary character satisfying $\chi(-I) = 1$ if $-I \in \Gamma^{\prime}$. If we set
\begin{equation}
\mathcal{A}(\Gamma^{\prime}, \chi) \coloneqq \{f: \mathbb{H}_{3} \rightarrow \mathbb{C}: \text{$f(\gamma w) = \chi(\gamma) f(w)$ for all $\gamma \in \Gamma^{\prime}$ and $w \in \mathbb{H}_{3}$} \},
\end{equation}
then a function $f \in \mathcal{A}(\Gamma^{\prime}, \chi)$ is said to be an automorphic form under $\Gamma^{\prime}$ with character $\chi$ if
\begin{itemize}
\item $f \in C^{\infty}(\mathbb{H}_{3})$ and is an eigenfunction of the Laplacian~\eqref{eq:Laplacian-3-dim}, namely
\begin{equation}
\Delta f = -\lambda_{f} f, \qquad \lambda_{f} = s_{f}(2-s_{f}),
\end{equation}
where the quantity $s_{f} \in \mathbb{C}$ is referred to as the spectral parameter of $f$;
\item $f$ has moderate growth at cusps, namely there exists a constant $A > 0$ such that
\begin{equation}
|f(w)| < (v+(1+|z|^2) v^{-1} )^{A}, \qquad w = (z, v) \in \mathbb{H}_{3}.
\end{equation}
\end{itemize}

Let $\mathcal{L}(\Gamma^{\prime}, \chi)$ denote the finite-dimensional Hilbert space of square-integrable automorphic forms under $\Gamma^{\prime}$ with character $\chi$, with the norm $\norm{\cdot}_{2}$ induced by the Petersson inner product
\begin{equation}
\langle f, g \rangle \coloneqq \int_{\Gamma^{\prime} \backslash \mathbb{H}_{3}} f(w) \overline{g(w)} \, d\mu(w).
\end{equation}
Two automorphic forms are identified if they are constant multiples of one another.

\subsection{Generalised Shimura correspondence}\label{subsect:Shimura-2}
Let the notation be as above, and let~$\chi_{3}(\gamma)$ denote the cubic Kubota character defined by~\eqref{eq:Kubota-symbol}. Furthermore, let $\Gamma_{2} \coloneqq \langle \mathrm{PSL}_{2}(\mathbb{Z}), \Gamma_{1}(3) \rangle$.
\begin{theorem}[{Cf.~\cites[Theorem~2.9]{Louvel2008}[Theorem~3.4]{Patterson1998}}]\label{thm:Shimura}
Keep the notation as above. There exists a bijective correspondence between $f \in \mathcal{L}(\Gamma, \mathbbm{1})$ and $\tilde{f} \in \mathcal{L}(\Gamma_{2}, \chi_{3})$. In particular, we have that $s_{f}-1 = \pm 3(s_{\tilde{f}}-1)$.
\end{theorem}

The former reference pertains to a broader setting in the spirit of~\cites[Theorem~5.3]{Flicker1980}.

\subsection{Metaplectic explicit formula}\label{subsect:proof-of-theorem-1-7}
The following explicit formula follows in an identical manner from the standard contour shift argument as in the proof of \cref{prop:spectral-explicit-formula,prop:explicit-formula}; cf.~\cites[Section~5]{Nakasuji2000}. More specifically, the argument relies on the Selberg trace formula of Friedman~\cites[Theorem~3.1]{Friedman2005}[Theorem~4.1.1]{Friedman2005-2}, which holds for any cofinite Kleinian group $\Gamma$ and admits twisting by any unitary representation $\chi$ of $\Gamma$. 
\begin{proposition}\label{prop:explicit-formula-metaplectic}
Keep the notation as above. Then we have for any $1 \leq T \leq x^{\frac{1}{6}}$ that
\begin{equation}\label{eq:explicit-formula-metaplectic}
\mathcal{E}_{\Gamma_{2}}^{(3)}(x, \chi_{3}) = \sum_{\pm} \mathop{\sum \sum}_{\substack{\tilde{f} \in \mathcal{L}(\Gamma_{2}, \chi_{3}) \\ 0 < t_{\tilde{f}} \leq T}} \frac{x^{1 \pm it_{\tilde{f}}}}{1 \pm it_{\tilde{f}}}+O_{\varepsilon} \Big(\frac{x^{\frac{4}{3}+\varepsilon}}{T} \Big).
\end{equation}
\end{proposition}

In alignment with the classical $3$-dimensional setting as addressed by Nakasuji, the natural limitation of the metaplectic explicit formula~\eqref{eq:explicit-formula-metaplectic} represents $\delta_{1}^{(3)}(\chi_{3}) \geq \frac{7}{6}$. Going beyond~this threshold necessitates a more sophisticated analysis based on the \textit{smooth} explicit formula \`{a} la Balog et al.~\cites[Lemma~4.1]{BalogBiroCherubiniLaaksonen2022}. Such investigations are reserved for future endeavours.

The following bound unconditionally generalises the second display on~\cites[Page~792]{Koyama2001}.
\begin{proposition}\label{prop:SES-metaplectic}
Let $X, T \geq 2$. Then
\begin{equation}
\mathop{\sum \sum}_{\substack{\tilde{f} \in \mathcal{L}(\Gamma_{2}, \chi_{3}) \\ 0 < t_{\tilde{f}} \leq T}} X^{it_{\tilde{f}}} \ll_{\varepsilon} T^{\frac{7}{4}+\varepsilon} X^{\frac{1}{12}+\varepsilon}+T^{2}.
\end{equation}
\end{proposition}

\begin{proof}
\cref{thm:Shimura} ensures the change of variables $X \mapsto X^{\frac{1}{3}}$ in~\cites[Equation~(3-48)]{Kaneko2022-2}~by substituting $\eta = 0$ from~\cites[Theorem~3.2]{Kaneko2025-2}. Note that the transition from $\mathbb{Z}[i]$ to $\mathbb{Z}[\omega]$ is straightforward in light of the Kuznetsov and pre-Kuznetsov formul{\ae} due to Kodama~\cites[Theorem~1.8]{Kodama2004}. This is the only occasion where the cubic structure is fundamentally invoked; otherwise the machinery ought to extend seamlessly to higher metaplectic coverings.
\end{proof}

\begin{proof}[Proof of \cref{thm:main-metaplectic}]
It follows from \cref{prop:explicit-formula-metaplectic,prop:SES-metaplectic} and partial summation that
\begin{equation}
\mathcal{E}_{\Gamma_{2}}^{(3)}(x, \chi_{3}) \ll_{\varepsilon} T^{\frac{3}{4}+\varepsilon} x^{\frac{13}{12}+\varepsilon}+Tx+\frac{x^{\frac{4}{3}+\varepsilon}}{T}.
\end{equation}
Optimising $T = x^{\frac{1}{7}}$ now justifies the desideratum.
\end{proof}

\subsection*{Acknowledgements}
We thank Dimitrios Chatzakos for valuable correspondence.


\newcommand{\etalchar}[1]{$^{#1}$}
\providecommand{\bysame}{\leavevmode\hbox to3em{\hrulefill}\thinspace}
\providecommand{\MR}{\relax\ifhmode\unskip\space\fi MR }
\providecommand{\MRhref}[2]{%
  \href{http://www.ams.org/mathscinet-getitem?mr=#1}{#2}
}
\providecommand{\Zbl}{\relax\ifhmode\unskip\space\thinspace\fi Zbl }
\providecommand{\MR}[2]{%
  \href{https://zbmath.org/?q=an:#1}{#2}
}
\providecommand{\doi}{\relax\ifhmode\unskip\space\thinspace\fi DOI }
\providecommand{\MR}[2]{%
  \href{https://doi.org/#1}{#2}
}
\providecommand{\SSNI}{\relax\ifhmode\unskip\space\thinspace\fi ISSN }
\providecommand{\MR}[2]{%
  \href{#1}{#2}
}
\providecommand{\ISBN}{\relax\ifhmode\unskip\space\thinspace\fi ISBN }
\providecommand{\MR}[2]{%
  \href{#1}{#2}
}
\providecommand{\arXiv}{\relax\ifhmode\unskip\space\thinspace\fi arXiv }
\providecommand{\MR}[2]{%
  \href{#1}{#2}
}
\providecommand{\href}[2]{#2}

\end{document}